\documentclass{amsart}
\usepackage{graphicx}
\usepackage{leftidx}
\usepackage{caption}
\newcommand {\scs}[1]{\scriptscriptstyle\scriptscriptstyle{#1}}
\newtheorem{thm}{Theorem}[section]

\newtheorem{lem}[thm]{Lemma}

\theoremstyle{definition}

\theoremstyle{remark}

\numberwithin{equation}{section}

\begin{document}

\title[On mixed steps-collocation schemes for nonlinear...]{On mixed steps-collocation schemes for nonlinear fractional delay differential equations}%
\author[Mousa-Abadian]{M.~Mousa-Abadian}\email{m.abadian@shahed.ac.ir}\email{momeni@shahed.ac.ir}
\author[Momeni-Masuleh]{S.~H.~Momeni-Masuleh$^*$}
\address{Department of Mathematics, Shahed University, P.O. Box 18151-159, Tehran, Iran}%
\thanks{$^*$Corresponding author}

\begin{abstract}
  This research deals with the numerical solution of non-linear fractional differential equations with delay using the method of steps and shifted Legendre (Chebyshev) collocation method. This article aims to present a new formula for the fractional derivatives (in the Caputo sense) of shifted Legendre polynomials. With the help of this tool and previous work of the authors, efficient numerical schemes for solving non-linear continuous fractional delay differential equations are proposed. The proposed schemes transform the nonlinear fractional delay differential equations to a non-delay one by employing the method of steps. Then, the approximate solution is expanded in terms of Legendre (Chebyshev) basis. Furthermore, the convergence analysis of the proposed schemes is provided. Several practical model examples are considered to illustrate the efficiency and accuracy of the proposed schemes.
\end{abstract}
\maketitle
\section{Introduction}
Delay differential equations (DDEs) belong to a broader class of functional differential equations in which the rate of change of the unknown function at a specific time is represented due to the values of the function at previous times. DDEs are also known as differential-difference equations. Laplace and Condorcet~\cite{chen2002} first studied these equations, and naturally, appear in various fields of science and engineering~\cite{2}. Fractional delay differential equations (FDDEs) are considered as a generalization of DDEs which contain derivatives of arbitrary fractional order. The integer order differential operator is a local operator, while the fractional order differential operator is a non-local operator. More precisely, the next state of a system which is modeled by fractional differential equations (FDEs), depends not only upon its present situation but also upon all of its past situations. The fractional order differential operator enables us to describe a real event more accurately than the classical integer order differential operator. In recent years, FDEs and FDDEs are frequently used to model many real phenomena in the fields of control theory~\cite{3}, biology~\cite{4,N4}, economy~\cite{5}, and so on.

Because of the computational complexities of fractional delay derivatives, for most of the FDEs and FDDEs, the exact solution is available and therefore, it is necessary to employ numerical methods for solving such equations. Shakeri and Dehghan~\cite{N1} employ the homotopy perturbation method to solve delay differential equations with integer order derivatives. Variational iteration method is considered by Saadatmandi and Dehghan~\cite{N2} to obtain the numerical solution of the generalized pantograph equation.  Moghaddam and Mostaghim~\cite{6,180} introduced  numerical methods in the frame work of the finite difference method for solving FDDEs. They also presented a matrix approach using fractional finite difference method to solve nonlinear FDDEs~\cite{7}. Prakash~et~al.~\cite{8} proposed a numerical algorithm based on a modified He-Laplace method to solve nonlinear FDDEs. Wang~\cite{9} combined the Adams-Bashforth-Moulton method with the linear interpolation method to find an approximate solution for FDDEs. Mohammed and Khlaif~\cite{a} applied the Adomian decomposition method to get the numerical solution of FDDEs. Mousa-Abadian and Momeni-Masuleh proposed a numerical scheme to solve linear fractional delay differential systems~\cite{Mousa}. Their scheme employs the method of steps to handle the delay term and Chebyshev-Tau method to construct the approximate solution. Sedaghat~et~al.~\cite{11} introduced a numerical scheme using Chebyshev polynomials for solving FDDEs of pantograph type. Saeed~et~al.~\cite{12} developed Chebyshev wavelet methods for solving FDDEs. Khader~\cite{13} derived an approximate formula of the Laguerre polynomials for the numerical treatment of FDDEs. Daftardar-Gejji~et~al.~\cite{16} extended a new predictor-corrector method to solve FDDEs. Moghaddam~et~al.~\cite{181} developed a numerical method based on the Adams-Bashforth-Moulton method for solving variable-order FDDEs. Yaghoobi~et~al.~\cite{182} devised a numerical scheme based on a cubic spline interpolation for solving variable-order FDDEs. Khosravian-Arab~et~al.~\cite{N5} devised new Lagrange basis functions to approximate fractional derivatives in unbounded domains. Their approach is based on the pseudo-spectral, Galerkin, and Petrov-Galerkin methods.

A numerical approach to solve nonlinear FDDEs can be a generalization of the method which introduced in Ref.~\cite{Mousa}. The Chebyshev collocation method can be considered solving nonlinear FDDEs. Of course, employing collocation methods are not restricted to use Chebyshev basis, but also Legendre basis can be applied. Therefore, a significant part of this article deals with solving nonlinear FDDEs by using Legendre basis. In this paper, we derive a new formula for the fractional derivatives of shifted Legendre polynomials and then present the efficient numerical schemes for solving nonlinear FDDEs.

The remainder of this article proceeds as follows. In the next section, the properties of shifted collocation-type bases are discussed. In Section~3, a formula for the fractional derivatives of shifted Legendre basis is derived. Section~4 describes mixed steps-collocation schemes for solving nonlinear FDDEs. The convergence analysis of the proposed schemes has been done in Section~5. Section~6 concerns with applying the proposed schemes to several nonlinear FDDEs. The conclusion is given in Section~7.

\section{Shifted collocation-type bases}
The most common collocation methods are those based on Chebyshev and Legendre basis.
Properties of shifted Chebyshev basis have been investigated in Ref.~\cite{Mousa}. Here, we deduce the properties of shifted Legendre basis. The Legendre basis $L_{\scs k}(t)$ for $k=0, 1, \dots,$  and $-1\leq t\leq 1$, can be defined as the solution of following ordinary differential equation~\cite{9a}
  \begin{equation*}
  \frac{d}{dt}\Big((1-t^2)\frac{dL_{\scs k}}{dt}(t)\Big)+k(k+1)L_{\scs k}(t)=0,
  \end{equation*}
  which satisfy $L_{\scs k}(\pm1)=(\pm1)^k$. For $k\geq 1$, we have the following recurrence formula
  \begin{equation}\label{rec}
  L_{\scs {k+1}}(t)=\frac{2k+1}{k+1}tL_{\scs k}(t)-\frac{k}{k+1}L_{\scs {k-1}}(t),
  \end{equation}
  where $L_{\scs 0}(t)=1$ and $L_{\scs 1}(t)=t$. The shifted Legendre basis are defined on the interval $[\alpha,\beta]$ using the change of variable
$t=\frac{2}{\beta-\alpha}(x-\beta)+1$.
For the simplicity, let us to denote $L_{\scs {k}}(\frac{2}{\beta-\alpha}(x-\beta)+1)$ by $L_{\scs{\alpha,\beta,k}}(x)$. Therefore, similar to \eqref{rec}, the following recurrence relation can be obtained
  \begin{equation*}
  L_{\scs{\alpha,\beta,k+1}}(x)=\frac{2k+1}{k+1}\big(\frac{2}{\beta-\alpha}(x-\beta)+1 \big)L_{\scs{\alpha,\beta,k}}(x)-\frac{k}{k+1}L_{\scs{\alpha,\beta,k-1}}(x).
  \end{equation*}
The shifted Legendre basis $L_{\scs{\alpha,\beta,k}}(x)$ can be presented in the following form
\begin{equation}\label{01}
L_{\scs{\alpha,\beta,n}}(x)=\frac{1}{2^n}\sum_{k=0}^{[\frac{n}{2}]}(-1)^k\binom{n}{k}\binom{2n-2k}{n}\big(\frac{2}{\beta-\alpha}(x-\beta)+1 \big)^{n-2k}.
\end{equation}
By using the identity
\begin{equation}\label{001}
\big(\frac{2}{\beta-\alpha}(x-\beta)+1 \big)^{n-2k}=\sum_{l=0}^{n-2k}\sum_{j=0}^{l}\frac{(-1)^{l-j}\binom{l}{j}\binom{n-2k}{l}2^l\beta^{l-j}}{(\beta-\alpha)^l}x^j,
\end{equation}
the shifted Legendre basis $L_{\scs{\alpha,\beta,k}}(x)$ can be written in terms of a power series in $x$ as
\begin{equation}\label{1}
L_{\scs{\alpha,\beta,n}}(x)=\sum_{k=0}^{[\frac{n}{2}]}\sum_{l=0}^{n-2k}\sum_{j=0}^{l}\frac{(-1)^{k-j+l}2^{l-n}\beta^{l-j}(2n-2k)!}{(\beta-\alpha)^l(n-k)!k!(n-2k-l)!j!(l-j)!}x^j,
\end{equation}
which satisfy $L_{\scs{\alpha,\beta,k}}(\beta)=1$ and $L_{\scs{\alpha,\beta,k}}(\alpha)=(-1)^k$.

The next lemma describes integer order derivatives of the shifted Legendre basis.
\begin{lem}\label{marz}
For $m\leq j\leq n$, we have
\begin{equation}
D^m L_{\scs{\alpha,\beta,n}}(x)=
\sum_{k=0}^{[\frac{n}{2}]}\sum_{l=m}^{n-2k}\sum_{j=m}^{l}\frac{(-1)^{k-j+l}2^{l-n}(\beta-\alpha)^{-l}\beta^{l-j}(2n-2k)!}{(n-k)!k!(n-2k-l)!(l-j)!(j-m)!}x^{j-m}.
\end{equation}
\end{lem}
\begin{proof}
The proof is easily obtained from Eq.~\eqref{1}.
\end{proof}
The shifted Legendre polynomials satisfy the following relation
\begin{equation}
\int_{\alpha}^\beta L_{\scs{\alpha,\beta,i}}(x) L_{\scs{\alpha,\beta,k}}(x)\mathrm{d}x=\delta_{ik}(k+\frac{1}{2})^{-1},
\end{equation}
i.e., they are orthogonal with each other concerning the unit weight function.

The shifted Legendre basis form an orthogonal system of polynomials, which is complete in the space of square integrable functions, i.e., ${\mathcal L}^2(\alpha, \beta)$. Therefore, any $u\in {\mathcal L}^2(\alpha,\beta)$ can be written as
  \begin{equation*}
  u(x)=\sum_{k=0}^\infty \hat{u}_{\scs k} L_{\scs{\alpha,\beta,k}}(x),
  \end{equation*}
   where
  \begin{equation*}
  \hat{u}_{\scs k}=(k+\frac{1}{2})\int_{\alpha}^\beta u(x)L_{\scs{\alpha,\beta,k}}(x)\mathrm{d}x, \quad k\geq 0.
  \end{equation*}\\
 The associated inner product and norm are given by
  \begin{equation*}
  (f,g)=\int_{\alpha}^\beta f(x)g(x)\mathrm{d}x
  \end{equation*}
and
\begin{equation*}
 ||f||^2_{{\mathcal L}^2(\alpha,\beta)}=(f,f)=\int_{\alpha}^{\beta}|f(t)|^2\mathrm{d}t.
\end{equation*}

We define $H^m(\alpha,\beta)$ to be the vector space of the functions $g\in {\mathcal L}^2(\alpha, \beta)$ such that all the distributional derivatives of $f$ of order up to $m$ can be represented by functions in ${\mathcal L}^2(\alpha, \beta)$. $H^m(\alpha,\beta)$ is endowed with the norm
\begin{equation*}
||f||^2_{H^m(\alpha,\beta)}=\sum_{k=0}^{m}\Big|\Big| \frac{\partial^k}{\partial x^k}f(x) \Big|\Big|^2_{{\mathcal L}^2(\alpha,\beta)}.
\end{equation*}
Furthermore, the associated semi-norm is defined as follows
\begin{equation*}
| f |^2_{H^{m:N}(\alpha,\beta)}=\sum_{j=\min(m,N)}^{m}|| f^{(j)}||^2_{{\mathcal L}^2(\alpha,\beta)},
\end{equation*}
where $N$ is the number of nodal bases.

  In the later sections of this paper, we will use the Gaussian integration formula to approximate integrals such as
\begin{equation}
\int_{\alpha}^{\beta}f(x)\mathrm{d}x.
\end{equation}
 Explicit formulas for the quadrature nodes and weights for discrete shifted Chebyshev and Legendre basis are~\cite{16}
\begin{itemize}
   \item Chebyshev Gauss-Lobatto\\
   For $j=0,1,\cdots,N$,
   \begin{equation}\label{nodeC}
      x_{\scs {\alpha, \beta, N, j}}=\frac{\beta-\alpha}{2}(x_{\scs{N, j}}-1)+\beta, w_{\scs {\alpha, \beta, N, j}}=
      \begin{cases}
        \frac{\pi}{2N},&j=0, N,\\
        \frac{\pi}{N},&j=1,\cdots,N-1,
      \end{cases}
   \end{equation}
  where
  \[
  x_{\scs{N, j}}=\cos\frac{j\pi}{N}.
  \]
   \item  Legendre  Gauss-Lobatto\\
   \begin{equation}\label{nodeL}
x_{\scs{\alpha, \beta, 0}}=\alpha,\, x_{\scs {\alpha, \beta, N}}=\beta,\, x_{\scs{\alpha, \beta, j}} (j=1, 2, \cdots, N-1){\mbox{  roots of }} L^{'}_{\scs{\alpha,\beta,N}}(x),
\end{equation}
and
\begin{equation*}
w_{\scs{\alpha,\beta,N,j}}=\frac{2}{N(N+1)[L_{\scs{\alpha,\beta,N}}(x_{\scs j})]^2},\quad j=0, 1, \cdots, N.
\end{equation*}
 \end{itemize}
 For any $p(x)\in\mathbb{P}_{2N+1}$, where $\mathbb{P}_{2N+1}$ is the space of polynomials of degree at most $2N+1$, we have
\begin{align*}
 \int_\alpha^\beta \frac{p(x)}{\sqrt{1-x^2}} \mathrm{d}x&=\sum_{j=0}^Np(x_{\scs j})w_{\scs{\alpha,\beta,N,j}},&(\text{Chebyshev Gauss-Lobatto}), \\
 {\mbox{and}}\quad \int_\alpha^\beta p(x)\mathrm{d}x&=\sum_{j=0}^Np(x_{\scs j})w_{\scs{\alpha,\beta,N,j}},&(\text{Legendre Gauss-Lobatto}).
\end{align*}

\section{Fractional derivatives of collocation bases}
The shifted Chebyshev basis' fractional derivatives have been discussed in Ref.~\cite{178}.
This section continues with obtaining a novel formula for fractional derivatives of shifted Legendre basis $L_{\scs{\alpha,\beta,n}}(x)$ in the Caputo sense~\cite{2p}. Similar to the shifted Chebyshev basis~\cite{Mousa} one can find the following lemma and theorem.
\begin{lem}
Let $\nu$ be a positive real number. Then fractional derivative of order $\nu$ of shifted Legendre polynomials $L_{\scs{\alpha,\beta,n}}(x)$ can be given by
\begin{equation}
D^\nu L_{\scs{\alpha,\beta,n}}(x)=0,\quad n=0, 1, \cdots, \lceil \nu \rceil-1.
\end{equation}
\end{lem}
\begin{thm}\label{th3}
The fractional derivative of order $\nu$ of the shifted Legendre basis is
\begin{equation}\label{3.2}
D^\nu L_{\scs{\alpha,\beta,n}}(x)=\sum_{i=0}^\infty S_\nu(n,i) L_{\scs{\alpha,\beta,i}}(x), \quad n=\lceil \nu \rceil, \lceil \nu \rceil+1, \cdots,
\end{equation}
where the ceiling function $\lceil \nu \rceil$ stands for the smallest integer greater than or equal to $\nu$ and
\begin{equation}\label{ss}
S_\nu(n,i)=\displaystyle\sum_{k=0}^{[\frac{n}{2}]}\sum_{l=\lceil \nu \rceil}^{n-2k}\sum_{j=\lceil \nu \rceil}^{l}\frac{(-1)^{k-j+l}2^{l-n}\beta^{l-j}(\beta-\alpha)^{-l}(2n-2k)!j!}{(n-k)!k!(n-2k-l)!j!(l-j)!\Gamma(j+1-\nu)}c_{ij},
\end{equation}
in which
\begin{equation}\label{cij}
c_{\scs{ij}}=(i+\frac{1}{2})\int_\alpha^\beta x^{j-\nu}L_{\scs{\alpha,\beta,i}}(x)\mathrm{d}x.
\end{equation}
\end{thm}
\begin{proof}
As we know, the Caputo fractional derivative of $x^k$ of order $\nu$ is
\[
D^\nu x^k =\left\{
\begin{array}{ll}
0,& k=0, 1, \cdots \text{ and} \ k<\lceil \nu \rceil,\\
\frac{\Gamma(k+1)}{\Gamma(k+1-\nu)}x^{k-\nu},&k=0, 1, \cdots \text{ and} \ k\geq \lceil \nu \rceil.
\end{array}\right.
\]
 Considering~\eqref{1}, for $n=\lceil \nu \rceil, \lceil \nu \rceil+1, \cdots,$ we obtain
\begin{align}
D^\nu L_{\scs{\alpha,\beta,n}}(x)=&\sum_{k=0}^{[\frac{n}{2}]}\sum_{l=0}^{n-2k}\sum_{j=0}^{l}\frac{(-1)^{k-j+l}2^{l-n}\beta^{l-j}(\beta-\alpha)^{-l}(2n-2k)!}{(n-k)!k!(n-2k-l)!j!(l-j)!}D^\nu x^j\nonumber \\
=&\sum_{k=0}^{[\frac{n}{2}]}\sum_{l=\lceil \nu \rceil}^{n-2k}\sum_{j=\lceil \nu \rceil}^{l}\frac{(-1)^{k-j+l}2^{l-n}\beta^{l-j}(\beta-\alpha)^{-l}(2n-2k)!}{(n-k)!k!(n-2k-l)!(l-j)!\Gamma(j+1-\nu)}x^{j-\nu}\nonumber.
\end{align}
By expanding $x^{j-\nu}$ in terms of the shifted Legendre basis, we arrive at the following form
\begin{equation*}
x^{j-\nu} = \sum_{i=0}^\infty c_{\scs{ij}}L_{\scs{\alpha,\beta,j}}(x),
\end{equation*}
where $c_{ij}$ is given in~\eqref{cij}, which completes the proof.
\end{proof}

\section{Mixed steps-collocation schemes}
In this section, we propose new numerical schemes based on the method of steps and Legendre (Chebyshev) collocation method to solve the following nonlinear FDDE
\begin{equation}\label{Asli}
\left\{
\begin{array}{lc}
\sum\limits_{j=0}^{n} A_ju^{(j)}(t)+D^\nu u(t)+\sum_{i=1}^l\lambda_iD^{\nu_i}u(t)=f(t,u(t),u(t-\tau)),&t\geq 0,\\
u(t)=\phi(t),& -\tau\leq t\leq 0,\\
u'(0)=\phi'(0)=\phi_1,&\\
\hspace{0.7cm}\vdots&\\
u^{(\mu-1)}(0)=\phi^{(\mu-1)}(0)=\phi_{\mu-1},&
\end{array}\right.
\end{equation}
where $\lambda_i$, $A_j\in\mathbb{R}$ are constants and $A_n\neq 0$, $0<\nu_1<\nu_2<\cdots<\nu_l<\nu$, $m-1<\nu\leq m$ are real constants, $D^\nu u(t)\equiv u^{\nu}(t)$ represents the Caputo fractional derivative of order $\nu$ of function $u(t)$, $\mu=\max\{m, n\}$ and the function $f$ is given  nonlinear continuous function in $u$ which satisfies the following Lipschitz conditions
\begin{align}
|f(t,y_1,u)-f(t,y_2,u)|&\leq L_1|y_1-y_2|\label{lip},\\
|f(t,y,u_1)-f(t,y,u_2)|&\leq L_2|u_1-u_2|.\label{lip2}
\end{align}
Also, throughout this paper, we shall assume the initial function $\phi(t)$ to be continuous on $[-\tau,0]$. These conditions guarantee the existence and uniqueness of the solution of problem~(\ref{Asli}) (see, e.g.,~\cite{17,18}).

Clearly, for $t\in[0,\tau]$, the nonlinear FDDE~(\ref{Asli}) is equivalent to the following nonlinear non-delay FDEs
\begin{equation}\label{Asli1}
\left\{
\begin{array}{lc}
\sum\limits_{j=0}^{n} A_ju^{(j)}(t)+D^\nu u(t)+\sum_{i=1}^l\lambda_iD^{\nu_i}u(t)=f(t,u(t),\phi(t-\tau)),&t\in (0,\tau],\\
u(0) = \phi(0)=\phi_0,&\\
u'(0)=\phi'(0)=\phi_1,&\\
\hspace{0.7cm}\vdots&\\
u^{(\mu-1)}(0)=\phi^{(\mu-1)}(0)=\phi_{\mu-1}.&
\end{array}\right.
\end{equation}
One way of solving the nonlinear FDE~(\ref{Asli1}) is to use shifted Chebyshev basis, which is an extension of the scheme presented in Ref.~\cite{Mousa}. Another way is to
expand the approximate solution $u_{\scs N}(t)$ in terms of truncated shifted Legendre basis. The latter idea leads to
\begin{equation}\label{3.3}
u_{\scs N}(t)=\sum_{k=0}^Na_{\scs k}L_{\scs{0,\tau ,k}}(t),\qquad  t\in[0,\tau],
\end{equation}
where $a_{\scs k}\in \mathbb{R}$ are the unknown coefficients that we aim to find them. Thanks to Theorem~\ref{th3}, we can express the derivatives $u^{(1)}(t)$, $u^{(2)}(t)$, $\cdots, u^{(n)}(t)$, $D^\nu u(t)$, $D^{\nu_1}u(t)$, $D^{\nu_2}u(t), \cdots, D^{\nu_l}u(t)$ in terms of unknowns coefficients $a_{\scs k}$.

Now, we employ the Legendre (Chebyshev) collocation method to solve \eqref{Asli1} numerically. To do this, the following equation
 \begin{equation}\label{Asli2}
\sum\limits_{j=0}^{n} A_ju{\scs N}^{(j)}(t)+D^\nu u_{\scs N}(t)+\sum_{i=1}^l\lambda_iD^{\nu_i}u_{\scs N}(t)=f(t,u_{\scs N}(t),\phi(t-\tau)),
\end{equation}
 must be satisfied at the shifted Legendre collocation nodes~\eqref{nodeL} (shifted Chebyshev collocation nodes~\eqref{nodeC}) exactly. In fact, by using \eqref{3.2} and \eqref{3.3}, for $j=0, 1, \cdots, N-\mu$, we get the following equations
 \begin{align}\label{Asli3}
\sum\limits_{j=0}^{n}\sum_{k=0}^N A_ja_{\scs k}L^{(j)}_{\scs{0,\tau ,k}}(t_{\scs{\alpha, \beta, j}})+ \sum_{k=0}^Na_{\scs k}D^\nu& L_{\scs{0,\tau ,k}}(t_{\scs{\alpha, \beta, j}})+\sum_{i=1}^l\sum_{k=0}^N\lambda_ia_{\scs k}D^{\nu_i}L_{\scs{0,\tau ,k}}(t_{\scs{\alpha, \beta, j}}) \nonumber\\
 =f(t_{\scs{\alpha, \beta, j}},\sum_{k=0}^Na_{\scs k}L_{\scs{0,\tau ,k}}&(t_{\scs{\alpha, \beta, j}}),\phi(t_{\scs{\alpha, \beta, j}}-\tau)),\ t_{\scs{\alpha, \beta, j}}\geq 0,
\end{align}
where $t_{\scs{\alpha, \beta, j}}$ are the same as $x_{\scs{\alpha, \beta, j}}$ which are defined by \eqref{nodeL}.
After imposing the initial conditions
\begin{equation}\label{3.31}
u_{\scs N}^{(i)}(0)=\phi_i,\quad i=0, 1, \cdots, \mu-1,
\end{equation}
we arrive at a nonlinear system of algebraic equations. Similarly, using Theorem~1 in Ref.~\cite{Mousa} and related shifted Chebyshev expansion, we obtain  an algebraic system of nonlinear equations. The nonlinear resulted systems can be solved, for example, by Newton's method. Therefore, the approximate solution $u_{\scs N}$ in the interval $[0, \tau]$ is now available.
To obtain the approximate solution of Eq.~\eqref{Asli} in $[\tau, 2\tau]$, the presented procedure is used. Generally, if we want to solve Eq.~\eqref{Asli} in the interval $[(i-1)\tau, i\tau]$, $i\geq 1$, we need to solve the following equation
\begin{equation}\label{aa1}
\sum\limits_{j=0}^{n} A_j\leftidx{_{\scs{(i)}}}u^{(j)}(t)+D^\nu \leftidx{_{\scs{(i)}}}u(t)+\sum_{p=1}^l\lambda_p D^{\nu_p}\leftidx{_{\scs{(i)}}}u(t)=f(t,\leftidx{_{\scs{(i)}}}u(t),\leftidx{_{\scs{(i-1)}}}u(t-\tau)),
\end{equation}
with the initial conditions
\[
\leftidx{_{\scs{(i)}}}u^{(j)}(0)=\leftidx{_{\scs{(i-1)}}}u^{(j)}(\tau), \quad j=0, 1, \cdots, \mu,
\]
where
\begin{equation*}
\leftidx{_{\scs{(i)}}}u(t)=u((i-1)\tau+t),\quad i\geq 1.
\end{equation*}
Now, using the proposed procedure, we get the approximate solution $\leftidx{_{\scs{(i)}}}u_{\scs N}$ of Eq.~\eqref{aa1}.
\section{Convergence analysis}
In this section, by a similar manner presented in Ref.~\cite{177}, we show that the obtained approximate solutions in the previous section are convergent to the exact solutions. In order to investigate the exponential rate of convergence of the proposed schemes, we consider the nonlinear FDDE~\eqref{aa1} on the interval $\Omega_i=[(i-1)\tau, i\tau]$.

Let us define $\leftidx{_{\scs{(i)}}}e_{\scs N}(t)=\leftidx{_{\scs{(i)}}}u_{\scs N}(t)-\leftidx{_{\scs{(i)}}}u(t)$ to be the error function of the proposed scheme, where $\leftidx{_{\scs{(i)}}}u(t)$ and $\leftidx{_{\scs{(i)}}}u_{\scs N}(t)$ are the exact and Legendre (Chebyshev) collocation solution of \eqref{aa1} at $i-th$ step, respectively. 

Hereafter, we use the subscript $w$ for the Legendre and Chebyshev weight functions.

The orthogonal projection operator $P_{\scs N}$ from ${\mathcal L}_{\scs w}^2(\Omega)$ onto $\mathbb{P}_{\scs N}$, where $\Omega=[\alpha, \beta]$, satisfies
\begin{equation*}
\forall \psi_N\in \mathbb{P}_{\scs N} , \quad \int_{0}^{\tau}(f-P_{\scs N}f)(r)dr=0,
\end{equation*}
for any function $f$ in ${\mathcal L}_{\scs w}^2(\Omega)$. It is clear that $P_{\scs N}$ belongs to $\mathbb{P}_{\scs N}$.

The following inequalities for the shifted Legendre (Chebyshev) polynomials and shifted Legendre (Chebyshev)-Gauss-Lobatto nodes for $k\geq 1$ can be obtained by a similar argument provided in Ref.~\cite{9a} 
 \begin{align}
   || y-P_{\scs N}(y)||_{H^{l}_w(\Omega)} \leq & CN^{2l-1/2-k}|y|_{H^{k:N}_w(\Omega)},\label{hef}
 \end{align}
where $y\in H_w^k(\Omega)$.

\begin{lem}[\cite{KK}]\label{KK1}
Let $F$ be a continuous function on the real line. Then the $n$-fold multiple integral of $F$ based at $a$ is given by
\[
\int_{a}^{t}\int_{a}^{t_n}\cdots \int_{a}^{t_2}F(t_1)\mathrm{d}T=\frac{1}{(n-1)!}\int_{a}^{t}(t-s)^{n-1}F(s)\mathrm{d}s,
\]
where $\mathrm{d}T=\mathrm{d}t_1\mathrm{d}t_2\cdots \mathrm{d}t_n$.
\end{lem}
Now, we present the convergence theorem of the proposed schemes.
\begin{thm}
Suppose that the exact solution $\leftidx{_{\scs{(i)}}}u(t)$ at $i-th$ step of Eq.~\eqref{aa1} is smooth enough, i.e. $\leftidx{_{\scs{(i)}}}u(t)\in H^k_w(\Omega)$ for $ i, k\geq 1$, and the corresponding mixed steps-collocation solution $\leftidx{_{\scs{(i)}}}u_{\scs N}(t)$ is given by shifted Legendre or Chebyshev basis. Then for sufficiently large $N$ we have
\begin{align}
||\leftidx{_{\scs{(i)}}}e_{\scs N}(t)||_{{\mathcal L^2_{\scs w}}(\Omega)}\leq &
 C_1N^{-k}|\leftidx{_{\scs{(i)}}}u|_{H_{\scs w}^{k:N}(\Omega)}+C_2 N^{-3/2}|\leftidx{_{\scs{(i)}}}u|_{H_{\scs w}^{1:N}(\Omega)}\nonumber\\
 &+C_3N^{2(\eta_1 -1)-1/2-k}|\leftidx{_{\scs{(i)}}}u|_{H_{\scs w}^{k:N}}\\
 &+C_{4}N^{2(\eta_2 -1)-1/2-k}|\leftidx{_{\scs{(i)}}}u|_{H_{\scs w}^{k:N}}+C_{5}N^{-3/2},\label{gha}\nonumber
\end{align}
where
 \begin{equation*}
\eta_1=
\begin{cases}
m,& m\leq 3,\\
3,& m>3,
\end{cases}\qquad
 \eta_2=
\begin{cases}
m_p,&m_p\leq 3,\\
3,& m_p>3,
\end{cases}
\end{equation*}
and the constants $C_i$ are independent of $N$ and only depend on $n, m$ and $\nu$.
\end{thm}
\begin{proof}
As $\leftidx{_{\scs{(i)}}}u_{\scs N}(t)$ is the mixed steps-collocation solution of Eq.~\eqref{aa1} on the interval $\Omega_i$, it satisfies the following equation
\begin{align*}
\sum_{j=0}^{n}A_j \leftidx{_{\scs{(i)}}}u_{\scs N}^{(j)}(t)+P_{\scs N}\big( \frac{1}{\Gamma(m-\nu)}&\int_{(i-1)\tau}^{t}(t-s)^{m-\nu-1}\leftidx{_{\scs{(i)}}}u_{\scs N}^{(m)}(s)\mathrm{d}s\big)\\
+P_{\scs N}\big( \sum_{p=1}^l\frac{\lambda_p}{\Gamma(m_p-\nu_p)}&\int_{(i-1)\tau}^t(t-s)^{m_p-\nu_p-1}\leftidx{_{\scs{(i)}}}u_{\scs N}^{(m_p)}(s)\mathrm{d}s \big)\\
=P_{\scs N}\big(f(t,\leftidx{_{\scs{(i)}}}u_{\scs N}(t),&\leftidx{_{(i-1)}}u_{\scs N}(t))\big).
\end{align*}
By $n$-times integration of the above expression, we have
\begin{align}
 \int_{(i-1)\tau}^{t}\int_{(i-1)\tau}^{t_n}\cdots \int_{(i-1)\tau}^{t_2}A_n\leftidx{_{\scs{(i)}}}u_{\scs N}^{(n)}(t_1)\mathrm{d}T&+\sum_{j=0}^{n-1}\int_{(i-1)\tau}^{t}\int_{(i-1)\tau}^{t_n}\cdots \int_{(i-1)\tau}^{t_2}A_j\leftidx{_{\scs{(i)}}}u_{\scs N}^{(j)}(t_1)\mathrm{d}T\nonumber  \\
  +\frac{1}{\Gamma(m-\nu)}\int_{(i-1)\tau}^{t}\int_{(i-1)\tau}^{t_n}\cdots \int_{(i-1)\tau}^{t_2}P_{\scs N}&\big(\int_{(i-1)\tau}^{t_1}(t_1-s)^{m-\nu-1}\leftidx{_{\scs{(i)}}}u_{\scs N}^{(m)}(s)\mathrm{d}s\big)\mathrm{d}T\nonumber\\
  +\sum_{p=1}^l\frac{\lambda_p}{\Gamma(m_p-\nu_p)}\int_{(i-1)\tau}^{t}\int_{(i-1)\tau}^{t_n}\cdots \int_{(i-1)\tau}^{t_2}&P_{\scs N}\big(\int_{(i-1)\tau}^{t_1}(t_1-s)^{m_p-\nu_p-1}\leftidx{_{\scs{(i)}}}u_{\scs N}^{(m_p)}(s)\mathrm{d}s\big)\mathrm{d}T\nonumber\\
  =\int_{(i-1)\tau}^{t}\int_{(i-1)\tau}^{t_n}\cdots \int_{(i-1)\tau}^{t_2}P_{\scs N}\big(f(t_1,\leftidx{_{\scs{(i)}}}u_{\scs N}(t_1),&\leftidx{_{(i-1)}}u_{\scs N}(t_1))\big)\mathrm{d}T.&\label{mul}
\end{align}
 Thanks to Lemma~\ref{KK1}, we can rewrite each of multiple integrals appeared in~\eqref{mul} as a single integral~\cite{177}
\begin{align}
A_n\leftidx{_{\scs{(i)}}}u_{\scs N}(t)+Q_0(t)+\sum_{j=0}^{n-1}\int_{(i-1)\tau}^{t}\frac{A_j}{(n-1-j)!}(t-s)^{n-1-j}&\leftidx{_{\scs{(i)}}}u_{\scs N}(s)\mathrm{d}s\nonumber\\
+\frac{1}{\Gamma(m-\nu)}\int_{(i-1)\tau}^{t}\frac{(t-s)^{n-1}}{(n-1)!}P_{\scs N}\big(\int_{(i-1)\tau}^{s}(s-s_1)^{m-\nu -1}&\leftidx{_{\scs{(i)}}}u_{\scs N}^{(m)}(s_1)\mathrm{d}s_1 \big)\mathrm{d}s\nonumber\\
+\sum_{p=1}^l\frac{\lambda_p}{\Gamma(m_p-\nu_p)}\int_{(i-1)\tau}^{t}\frac{(t-s)^{n-1}}{(n-1)!}P_{\scs N}\big(\int_{(i-1)\tau}^{s}(s-s_1)^{m_p-\nu_p -1}&\leftidx{_{\scs{(i)}}}u_{\scs N}^{(m_p)}(s_1)\mathrm{d}s_1 \big)\mathrm{d}s  \nonumber\\
=\int_{(i-1)\tau}^{t}\frac{(t-s)^{n-1}}{(n-1)!}P_{\scs N}\big(f(s,\leftidx{_{\scs{(i)}}}u_{\scs N}(s),\leftidx{_{(i-1)}}u_{\scs N}(s))\big)\mathrm{d}s,&\label{21}
\end{align}
where $Q_0(t)$ contains initial conditions. Similarly, the exact solution $\leftidx{_{\scs{(i)}}}u(t)$ satisfies the following relation
\begin{align}
A_n\leftidx{_{\scs{(i)}}}u(t)+Q_0(t)+\sum_{j=0}^{n-1}\int_{(i-1)\tau}^{t}\frac{A_j}{(n-1-j)!}(t-s)^{n-1-j}&\leftidx{_{\scs{(i)}}}u(s)\mathrm{d}s\nonumber\\
+\frac{1}{\Gamma(m-\nu)}\int_{(i-1)\tau}^{t}\frac{(t-s)^{n-1}}{(n-1)!}\int_{(i-1)\tau}^{s}(s-s_1)^{m-\nu -1}&\leftidx{_{\scs{(i)}}}u^{(m)}(s_1)\mathrm{d}s_1\mathrm{d}s\nonumber\\
+\sum_{p=0}^l\frac{\lambda_p}{\Gamma(m_p-\nu_p)}\int_{(i-1)\tau}^{t}\frac{(t-s)^{n-1}}{(n-1)!}\int_{(i-1)\tau}^{s}(s-s_1)^{m_p-\nu_p -1}&\leftidx{_{\scs{(i)}}}u^{(m_p)}(s_1)\mathrm{d}s_1\mathrm{d}s  \nonumber\\
=\int_{(i-1)\tau}^{t}\frac{(t-s)^{n-1}}{(n-1)!}f(s,\leftidx{_{\scs{(i)}}}u(s),\leftidx{_{(i-1)}}u(s))\mathrm{d}s.&\label{22}
\end{align}
Subtracting \eqref{22} from \eqref{21} leads to
\begin{align}
A_n\leftidx{_{\scs{(i)}}}e_{\scs N}(t)&+\sum_{j=0}^{n-1}\frac{A_j}{(n-1-j)!}\int_{(i-1)\tau}^{t}(t-s)^{n-1-j}\leftidx{_{\scs{(i)}}}e_{\scs N}(s)\mathrm{d}s\nonumber\\
&+\frac{1}{\Gamma(m-\nu)}\int_{(i-1)\tau}^{t}\frac{(t-s)^{n-1}}{(n-1)!}\leftidx{_{\scs{(i)}}}e_{P_{\scs N}}\mathrm{d}s   \nonumber\\
&+\frac{1}{\Gamma(m-\nu)}\int_{(i-1)\tau}^{t}\frac{(t-s)^{n-1}}{(n-1)!}\int_{(i-1)\tau}^{s}(s-s_1)^{m-\nu-1}\leftidx{_{\scs{(i)}}}e_{\scs N}^{(m)}(s_1)\mathrm{d}s_1\mathrm{d}s\nonumber\\
&+\sum_{p=1}^l\frac{\lambda_p}{\Gamma(m_p-\nu_p)}\int_{(i-1)\tau}^{t}\frac{(t-s)^{n-1}}{(n-1)!}\leftidx{_{\scs{(i)}}}e_{P^p_N}\mathrm{d}s\nonumber\\
&\hspace{-1cm}+\sum_{p=1}^l\frac{\lambda_p}{\Gamma(m_p-\nu_p)}\int_{(i-1)\tau}^{t}\frac{(t-s)^{n-1}}{(n-1)!}\int_{(i-1)\tau}^{s}(s-s_1)^{m_p-\nu_p -1}\leftidx{_{\scs{(i)}}}e_{\scs N}^{(m_p)}(s_1)\mathrm{d}s_1\mathrm{d}s\nonumber\\
&=\int_{(i-1)\tau}^{t}\frac{(t-s)^{n-1}}{(n-1)!}\leftidx{_{\scs{(i)}}}e_f(s)\mathrm{d}s,\label{si7}
\end{align}
where
\begin{align*}
\leftidx{_{\scs{(i)}}}e_f(s)&=P_{\scs N}(f(s,\leftidx{_{\scs{(i)}}}u_{\scs N}(s),\leftidx{_{(i-1)}}u_{\scs N}(s)))-f(s,\leftidx{_{\scs{(i)}}}u(s),\leftidx{_{(i-1)}}u(s)),\\
\leftidx{_{\scs{(i)}}}e_{\scs P_{\scs N}}(s)&=P_{\scs N}\big( \int_{(i-1)\tau}^{s}(s-s_1)^{m-\nu-1}\leftidx{_{\scs{(i)}}}u_{\scs N}^{(m)}(s_1)\mathrm{d}s_1\big)\\
&\qquad-\int_{(i-1)\tau}^{s}(s-s_1)^{m-\nu-1}\leftidx{_{\scs{(i)}}}u_{\scs N}^{(m)}(s_1)\mathrm{d}s_1,\nonumber
\end{align*}
and
\begin{align*}
\leftidx{_{\scs{(i)}}}e_{P^p_N}(s)&=P_{\scs N}\big( \int_{(i-1)\tau}^{s}(s-s_1)^{m_p-\nu_p -1}\leftidx{_{\scs{(i)}}}u_{\scs N}^{(m_p)}(s_1)\mathrm{d}s_1\big)\\
&\qquad-\int_{(i-1)\tau}^{s}(s-s_1)^{m_p-\nu_p -1}\leftidx{_{\scs{(i)}}}u_{\scs N}^{(m_p)}(s_1)\mathrm{d}s_1.
\end{align*}
After $(n-1)$-times integrating by parts of the fourth and sixth term of the left-hand side of~\eqref{si7} we get (see Apendix~\ref{Ap1})
\begin{align}
\frac{1}{\Gamma(m-\nu)}\int_{(i-1)\tau}^{t}\frac{(t-s)^{n-1}}{(n-1)!}&\int_{(i-1)\tau}^{s}(s-s_1)^{m-\nu-1}\leftidx{_{\scs{(i)}}}e_{\scs N}^{(m)}(s_1)\mathrm{d}s_1\mathrm{d}s\nonumber\\
=\frac{1}{\Gamma(m-\nu)\prod_{j=1}^{n-1}(m-\nu+j-1)}&\int_{(i-1)\tau}^{t}\int_{(i-1)\tau}^{s}(s-s_1)^{m+n-\nu-2}\leftidx{_{\scs{(i)}}}e_{\scs N}^{(m)}(s_1)\mathrm{d}s_1\mathrm{d}s\label{si3},
\end{align}
and
\begin{align}
\frac{1}{\Gamma(m_p-\nu_p)}\int_{(i-1)\tau}^{t}\frac{(t-s)^{n-1}}{(n-1)!}&\int_{(i-1)\tau}^{s}(s-s_1)^{m_p-\nu_p-1}\leftidx{_{\scs{(i)}}}e_{\scs N}^{(m_p)}(s_1)\mathrm{d}s_1\mathrm{d}s\nonumber\\
=\frac{1}{\Gamma(m_p-\nu_p)\prod_{j=1}^{n-1}(m_p-\nu_p+j-1)}&\int_{(i-1)\tau}^{t}\int_{(i-1)\tau}^{s}(s-s_1)^{m_p+n-\nu_p-2}\leftidx{_{\scs{(i)}}}e_{\scs N}^{(m_p)}(s_1)\mathrm{d}s_1\mathrm{d}s\label{si4}.
\end{align}
Now we consider the three different cases:
 \begin{description}
   \item [{\it(i)}] $m\geq 4$ and $ m_p\geq 4$,
   \item [{\it(ii)}] $m\geq 4$ and $ m_p\leq 4$,
   \item [{\it(iii)}] $m\leq 4$.
 \end{description}

 Case $(i)$: $(m-3)$-times integrating by parts of right-hand side of equations~\eqref{si3} and~\eqref{si4}, for $n+1>\nu$, gives
\begin{align}
\frac{1}{\Gamma(m-\nu)\prod_{j=1}^{n-1}(m-\nu+j-1)}\int_{(i-1)\tau}^{t}\int_{(i-1)\tau}^{s}(s-s_1)^{m+n-\nu-2}&\leftidx{_{\scs{(i)}}}e_{\scs N}^{(m)}(s_1)\mathrm{d}s_1\mathrm{d}s\nonumber\\
=\frac{\prod_{j=1}^{m-3}(m+n-\nu-1-j)}{\Gamma(m-\nu)\prod_{j=1}^{n-1}(m-\nu+j-1)}\int_{(i-1)\tau}^{t}\int_{(i-1)\tau}^{s}(s-s_1)^{n-\nu+1}&\leftidx{_{\scs{(i)}}}e_{\scs N}^{(3)}(s_1)\mathrm{d}s_1\mathrm{d}s,\label{si5}
\end{align}
and
\begin{align}
\frac{1}{\Gamma(m_p-\nu_p)\prod_{j=1}^{n-1}(m_p-\nu_p+j-1)}\int_{(i-1)\tau}^{t}\int_{(i-1)\tau}^{s}(s-s_1)^{m_p+n-\nu_p -2}\leftidx{_{\scs{(i)}}}e_{\scs N}^{(m_p)}(s_1)\mathrm{d}s_1\mathrm{d}s&\nonumber\\
=\frac{\prod_{j=1}^{m_p-3}(m_p+n-\nu_p-1-j)}{\Gamma(m_p-\nu_p)\prod_{j=1}^{n-1}(m_p-\nu_p+j-1)}\int_{(i-1)\tau}^{t}\int_{(i-1)\tau}^{s}(s-s_1)^{n-\nu_p+1}\leftidx{_{\scs{(i)}}}e_{\scs N}^{(3)}(s_1)\mathrm{d}s_1\mathrm{d}s.&\label{si6}
\end{align}
Substituting the right-hand side of~\eqref{si5} into the right-hand side of~\eqref{si3}, we obtain
\begin{align}
\frac{1}{\Gamma(m-\nu)}\int_{(i-1)\tau}^{t}\frac{(t-s)^{n-1}}{(n-1)!}\int_{(i-1)\tau}^{s}(s-s_1)^{m-\nu-1}\leftidx{_{\scs{(i)}}}e_{\scs N}^{(m)}(s_1)&\mathrm{d}s_1\mathrm{d}s   \nonumber\\
=\frac{\prod_{j=1}^{m-3}(m+n-\nu-1-j)}{\Gamma(m-\nu)\prod_{j=1}^{n-1}(m-\nu+j-1)}\int_{(i-1)\tau}^{t}\int_{(i-1)\tau}^{s}(s-s_1)^{n-\nu+1}&\leftidx{_{\scs{(i)}}}e_{\scs N}^{(3)}(s_1)\mathrm{d}s_1\mathrm{d}s,\label{a23}
\end{align}
and similarly, from equations~\eqref{si4} and~\eqref{si6} we have
\begin{align}
\frac{1}{\Gamma(m_p-\nu_p)}\int_{(i-1)\tau}^{t}\frac{(t-s)^{n-1}}{(n-1)!}\int_{(i-1)\tau}^{s}(s-s_1)^{m_p-\nu_p -1}\leftidx{_{\scs{(i)}}}e_{\scs N}^{(m_p)}(s_1)\mathrm{d}s_1\mathrm{d}s&\nonumber\\
=\frac{\prod_{j=1}^{m_p-3}(m_p+n-\nu_p-1-j)}{\Gamma(m_p-\nu_p)\prod_{j=1}^{n-1}(m_p-\nu_p+j-1)}\int_{(i-1)\tau}^{t}\int_{(i-1)\tau}^{s}(s-s_1)^{n-\nu_p+1}\leftidx{_{\scs{(i)}}}e_{\scs N}^{(3)}(s_1)\mathrm{d}s_1\mathrm{d}s.&\label{a244}
\end{align}
Substituting~\eqref{a23} and~\eqref{a244} into~\eqref{si7}, we arrive at
\begin{align}
|\leftidx{_{\scs{(i)}}}e_{\scs N}(t)|\leq & \sum_{j=0}^{n-1}\Big|\frac{A_j}{A_n(n-1-j)!}\Big|\int_{(i-1)\tau}^{t}\Big|(t-s)^{n-1-j}\leftidx{_{\scs{(i)}}}e_{\scs N}(s) \Big|\mathrm{d}s\nonumber\\
&+\Big| \frac{1}{A_n(n-1)!\Gamma(m-\nu)}\Big|  \int_{(i-1)\tau}^{t}\Big|(t-s)^{n-1}\leftidx{_{\scs{(i)}}}e_{\scs P_{\scs N}}(s)\Big|\mathrm{d}s  \nonumber\\
&+\Big|\frac{\prod_{j=1}^{m-3}(m+n-\nu-1-j)}{A_n\Gamma(m-\nu)\prod_{j=1}^{n-1}(m-\nu+j-1)}\Big|\int_{(i-1)\tau}^{t}\int_{(i-1)\tau}^{s}\Big|(s-s_1)^{n-\nu+1}\leftidx{_{\scs{(i)}}}e_{\scs N}^{(3)}(s_1)\Big|\mathrm{d}s_1\mathrm{d}s\nonumber\\
&+\sum_{p=1}^l\Big|\frac{\lambda_p}{A_n\Gamma(m_p-\nu_p)}\Big| \int_{(i-1)\tau}^{t}\Big| (t-s)^{n-1}\leftidx{_{\scs{(i)}}}e_{P^p_N}(s)\Big|\mathrm{d}s \nonumber\\
 &+\sum_{p=1}^l\Big|\frac{\prod_{j=1}^{m_p-3}(m_p+n-\nu_p-1-j)}{A_n\Gamma(m_p-\nu_p)\prod_{j=1}^{n-1}(m_p-\nu_p+j-1)}\Big|\int_{(i-1)\tau}^{s}\Big|(s-s_1)^{n-\nu_p+1}\leftidx{_{\scs{(i)}}}e_{\scs N}^{(3)}(s_1)\Big|\mathrm{d}s_1\mathrm{d}s\nonumber\\
 &+\int_{(i-1)\tau}^{t}\Big|\frac{(t-s)^{n-1}}{A_n(n-1)!}\leftidx{_{\scs{(i)}}}e_f(s)\Big|\mathrm{d}s. \label{bis07}
 \end{align}
 We may rewrite Eq.~\eqref{bis07} as
\begin{align}
|\leftidx{_{\scs{(i)}}}e_{\scs N}(t)|\leq  & \gamma_1\int_{(i-1)\tau}^{t}|\leftidx{_{\scs{(i)}}}e_{\scs N}(s)|\mathrm{d}s+\gamma_2\int_{(i-1)\tau}^{t}|\leftidx{_{\scs{(i)}}}e_{\scs P_{\scs N}}(s)|\mathrm{d}s+\gamma_3\int_{(i-1)\tau}^{t}\Big|s^{n-\nu+1}\leftidx{_{\scs{(i)}}}e_{\scs N}^{''}(s)\Big|\mathrm{d}s\nonumber\\
&+\gamma_4\sum_{p=1}^l\int_{(i-1)\tau}^{t}|\leftidx{_{\scs{(i)}}}e_{P^p_N}(s)|\mathrm{d}s+\gamma_5\sum_{p=1}^l\int_{(i-1)\tau}^{t}\Big|s^{n-\nu_p+1}\leftidx{_{\scs{(i)}}}e_{\scs N}^{''}(s)\Big|\mathrm{d}s\nonumber\\
&\qquad+\gamma_6\int_{(i-1)\tau}^{t}|\leftidx{_{\scs{(i)}}}e_f(s)|\mathrm{d}s,\label{bis7}
\end{align}
where $\gamma_i$ are independent of $N$ and only may depend on $n$, $m$, $m_p$ and $\nu$.

By the Gronwall lemma~\cite{9} we get
\begin{align}
|\leftidx{_{\scs{(i)}}}e_{\scs N}(t)|\leq& e^{(\int_{(i-1)\tau}^{t}(\gamma_1)\mathrm{d}s }  \Big( \gamma_2\int_{(i-1)\tau}^{t}|\leftidx{_{\scs{(i)}}}e_{\scs P_{\scs N}}(s)|\mathrm{d}s+\gamma_3\int_{(i-1)\tau}^{t}\Big|s^{n-\nu+1}\leftidx{_{\scs{(i)}}}e_{\scs N}^{''}(s)\Big|\mathrm{d}s\nonumber\\
&+\gamma_4\sum_{p=1}^l\int_{(i-1)\tau}^{t}|\leftidx{_{\scs{(i)}}}e_{P^p_N}(s)|\mathrm{d}s+\gamma_5\sum_{p=1}^l\int_{(i-1)\tau}^{t}\Big|s^{n-\nu_p+1}\leftidx{_{\scs{(i)}}}e_{\scs N}^{''}(s)\Big|\mathrm{d}s\nonumber\\
&+\gamma_6\int_{(i-1)\tau}^{t}|\leftidx{_{\scs{(i)}}}e_f(s)|\mathrm{d}s \Big)\nonumber\\
\leq &\gamma_7\int_{(i-1)\tau}^{t}|\leftidx{_{\scs{(i)}}}e_{\scs P_{\scs N}}(s)|\mathrm{d}s+\gamma_8\int_{(i-1)\tau}^{t}\Big|s^{n-\nu+1}\leftidx{_{\scs{(i)}}}e_{\scs N}^{''}(s)\Big|\mathrm{d}s\nonumber\\
&+\gamma_{9}\sum_{p=1}^l\int_{(i-1)\tau}^{t}|\leftidx{_{\scs{(i)}}}e_{P^p_N}(s)|\mathrm{d}s+\gamma_{10}\sum_{p=1}^l\int_{(i-1)\tau}^{t}\Big|s^{n-\nu_p+1}\leftidx{_{\scs{(i)}}}e_{\scs N}^{''}(s)\Big|\mathrm{d}s\nonumber\\
&\qquad+\gamma_{11}\int_{(i-1)\tau}^{t}|\leftidx{_{\scs{(i)}}}e_f(s)|\mathrm{d}s,\nonumber
\end{align}
where $\gamma_7,\cdots,\gamma_{11}$ are some constants which related to $\gamma_1,\cdots,\gamma_6$.

From Lipschitz conditions~\eqref{lip} and~\eqref{lip2}, inequality~\eqref{hef} and generalized Hardy's inequality~\cite{HH} we obtain
\begin{align}
||\leftidx{_{\scs{(i)}}}e_{\scs N}(t)||_{{\mathcal L^2_{\scs w}}(\Omega)}\leq&\gamma_{12} ||\leftidx{_{\scs{(i)}}}e_{\scs P_{\scs N}}(s)||_{{\mathcal L^2_{\scs w}}(\Omega)}+\gamma_{13} || s^{n-\nu+1}\leftidx{_{\scs{(i)}}}e_{\scs N}^{''}(s)||_{{\mathcal L^2_{\scs w}}(\Omega)}\nonumber\\
&+\gamma_{14}\sum_{p=1}^l||\leftidx{_{\scs{(i)}}}e_{P^p_N}(s)||_{{\mathcal L^2_{\scs w}}(\Omega)}\label{bis8}\nonumber\\
&+\gamma_{15}\sum_{p=1}^l|| s^{n-\nu_p +1}\leftidx{_{\scs{(i)}}}e_{\scs N}^{''}(s)||_{{\mathcal L^2_{\scs w}}(\Omega)}+C_{5}N^{-3/2},
\end{align}
where $\gamma_{12},\cdots,\gamma_{15}$ are some constants which related to $\gamma_7,\cdots,\gamma_{11}$ and are independent of $N$.
From \eqref{hef} we get
\begin{align}
||s^{n-\nu}\leftidx{_{\scs{(i)}}}e_{\scs N}^{''}(s)||_{{\mathcal L^2_{\scs w}}(\Omega)}&\leq ||s^{n-\nu}||_{{\mathcal L^2_{\scs w}}(\Omega)}||\leftidx{_{\scs{(i)}}}e_{\scs N}^{''}(s)||_{{\mathcal L^2_{\scs w}}(\Omega)}\nonumber\\
&\leq \gamma_{16} ||\leftidx{_{\scs{(i)}}}e_{\scs N}(s)||_{H^2_w(\Omega)}\nonumber\\
&\leq \gamma_{17}N^{7/2-k}|\leftidx{_{\scs{(i)}}}u|_{H^{k:N}_w(\Omega)},\label{bis9}
\end{align}
and
\begin{align}
||s^{n-\nu_p}\leftidx{_{\scs{(i)}}}e_{\scs N}^{''}(s)||_{{\mathcal L^2_{\scs w}}(\Omega)}&\leq ||s^{n-\nu_p}||_{{\mathcal L^2_{\scs w}}(\Omega)}||\leftidx{_{\scs{(i)}}}e_{\scs N}^{''}(s)||_{{\mathcal L^2_{\scs w}}(\Omega)}\nonumber\\
&\leq \gamma_{18} ||\leftidx{_{\scs{(i)}}}e_{\scs N}(s)||_{H^2_w(\Omega)}\nonumber\\
&\leq \gamma_{19}N^{7/2-k}|\leftidx{_{\scs{(i)}}}u|_{H^{k:N}_w(\Omega)},\label{bis10}
\end{align}
where $\gamma_{16},\cdots,\gamma_{19}$ do not depend on $N$.
From \eqref{hef} we have
\begin{align}
||\leftidx{_{\scs{(i)}}}e_{P_{\scs N}}(t)||_{{\mathcal L^2_{\scs w}}(\Omega)}&\leq C_6N^{-3/2}\Big|\int_{(i-1)\tau}^{t}(t-s)^{m-\nu-1}\leftidx{_{\scs{(i)}}}u_{\scs N}^{(m)}(s)\mathrm{d}s\Big|_{H^{1:N}_w(\Omega)}\nonumber\\
&=C_7N^{-3/2}|D^\nu \leftidx{_{\scs{(i)}}}u_{\scs N}|_{H^{1:N}_w(\Omega)}\label{mosh},\\
||\leftidx{_{\scs{(i)}}}e_{P^p_N}(t)||_{{\mathcal L^2_{\scs w}}(\Omega)}&\leq C_8N^{-3/2}\Big|\int_{(i-1)\tau}^{t}(t-s)^{m_p-\nu_p-1}\leftidx{_{\scs{(i)}}}u_{\scs N}^{(m_p)}(s)\mathrm{d}s\Big|_{H^{1:N}_w(\Omega)}\nonumber\\
&=C_9N^{-3/2}|D^{\nu_p} \leftidx{_{\scs{(i)}}}u_{\scs N}|_{H^{1:N}_w(\Omega)}\label{mosh1}.
\end{align}
Linear operators $D^\nu :\mathbb{P}_N\to \mathbb{P}_N^\nu$ and $D^{\nu_p} :\mathbb{P}_N\to \mathbb{P}_N^{\nu_p}$ are bounded (see Appendix~\ref{Ap3}), so that the constants $C_{10}$  and $C_{11}^p$ can be found such that
\begin{equation}
|D^\nu \leftidx{_{\scs{(i)}}}u_{\scs N}|_{H^{k:N}_w(\Omega)}\leq C_{10}| \leftidx{_{\scs{(i)}}}u_{\scs N}|_{H^{k:N}_w(\Omega)}\label{bou1},
 \end{equation}
 and
 \begin{equation}
|D^{\nu_p} \leftidx{_{\scs{(i)}}}u_{\scs N}|_{H^{k:N}_w(\Omega)}\leq C_{11}^p| \leftidx{_{\scs{(i)}}}u_{\scs N}|_{H^{k:N}_w(\Omega)}\label{bou2}.
 \end{equation}
  Therefore, from~\eqref{mosh} and~\eqref{bou1} we have
\begin{align}
||\leftidx{_{\scs{(i)}}}e_{\scs P_{\scs N}}(t)||_{{\mathcal L^2_{\scs w}}(\Omega)}&\leq C_{12}N^{-3/2}|\leftidx{_{\scs{(i)}}}u_{\scs N}(s)|_{H^{1:N}_w(\Omega)}\nonumber\\
&=C_{12}N^{-3/2}|\leftidx{_{\scs{(i)}}}u_{\scs N}(s)-\leftidx{_{\scs{(i)}}}u(s)+\leftidx{_{\scs{(i)}}}u(s)|_{H^{1:N}_w(\Omega)}\nonumber\\
&\leq C_{12}N^{-3/2}\big( |\leftidx{_{\scs{(i)}}}e_{\scs N}(s)|_{H^{1:N}_w(\Omega)}+|\leftidx{_{\scs{(i)}}}u(s)|_{H^{1:N}_w(\Omega)} \big).
\end{align}
Since $u_{\scs N}=P_{\scs N}u$, we can write
\begin{align}
|\leftidx{_{\scs{(i)}}}e_{\scs N}(s)|_{H^{1:N}_w(\Omega)}=&|\leftidx{_{\scs{(i)}}}u_{\scs N}(s)-\leftidx{_{\scs{(i)}}}u(s)+P_{\scs N}\leftidx{_{\scs{(i)}}}u(s)-P_{\scs N}\leftidx{_{\scs{(i)}}}u(s)|_{H^{1:N}_w(\Omega)}\nonumber\\
=&|\leftidx{_{\scs{(i)}}}u(s)-P_{\scs N}\leftidx{_{\scs{(i)}}}u(s)|_{H^{1:N}_w(\Omega)}.\nonumber
\end{align}
So that
\begin{equation}\label{si}
||\leftidx{_{\scs{(i)}}}e_{\scs P_{\scs N}}(t)||_{{\mathcal L^2_{\scs w}}(\Omega)}{\overset{\scs{ \eqref{hef}}}{\leq}} C_{13}N^{-k} |\leftidx{_{\scs{(i)}}}u|_{H^{k:N}_w(\Omega)}+C_{12}N^{-3/2}|\leftidx{_{\scs{(i)}}}u|_{H^{1:N}_w(\Omega)}.
\end{equation}
From~\eqref{mosh1} and~\eqref{bou2} we have
\begin{align}
||\leftidx{_{\scs{(i)}}}e_{P^p_N}(t)||_{{\mathcal L^2_{\scs w}}(\Omega)}&\leq C_{14}^pN^{-3/2}|\leftidx{_{\scs{(i)}}}u_{\scs N}(s)|_{H^{1:N}_w(\Omega)}\nonumber\\
&\leq C_{14}^pN^{-3/2}\big( |\leftidx{_{\scs{(i)}}}e_{\scs N}(s)|_{H^{1:N}_w(\Omega)}+|\leftidx{_{\scs{(i)}}}u|_{H^{1:N}_w(\Omega)} \big)\nonumber\\
&{\overset{\scs{ \eqref{hef}}}{\leq}} C_{15}^pN^{-k} |\leftidx{_{\scs{(i)}}}u|_{H^{k:N}_w(\Omega)}+C_{14}^pN^{-3/2}|\leftidx{_{\scs{(i)}}}u|_{H^{1:N}_w(\Omega)}.\label{si1}
\end{align}
From Eqs. \eqref{bis8}-\eqref{si1}, for $m\geq 4$ and $ m_p\geq 4$, we have
\begin{align*}
||\leftidx{_{\scs{(i)}}}e_{\scs N}(t)||_{{\mathcal L^2_{\scs w}}(\Omega)}\leq &
 C_1N^{-k}|\leftidx{_{\scs{(i)}}}u|_{H_{\scs w}^{k:N}(\Omega)}+C_2 N^{-3/2}|\leftidx{_{\scs{(i)}}}u|_{H_{\scs w}^{1:N}(\Omega)}\\
 &+C_3N^{2(\eta_1 -1)-1/2-k}|\leftidx{_{\scs{(i)}}}u|_{H_{\scs w}^{k:N}}\\
 &+C_{4}N^{2(\eta_2 -1)-1/2-k}|\leftidx{_{\scs{(i)}}}u|_{H_{\scs w}^{k:N}}+C_{5}N^{-3/2},   
\end{align*}
where $C_1,\cdots, C_5$ are constants which related to earlier $\gamma_i$'s and $C_i$'s.

The same argument can be done for the cases $(ii)$ and $(iii)$ by assuming that $n+1>\nu$, $m_p-\nu_p+n>2$, and $m-\nu+n>2$ respectively.
\end{proof}

\section{Numerical results}
In this section, we consider several practical examples which, in general, do not have an exact solution. Here, we made use of the $Matlab~2014a$ package. The computational codes were conducted on an Intel(R) Core(TM) i7-6700K processor, equipped with 8 GB of RAM. Also, We use the fix-point iteration method for solving nonlinear systems, and stopping criterion is set to be $10^{-15}$.

\textbf{Example 1.} Consider the following FDDE
\begin{equation}\label{Ex3}
D^{\nu}u(t)=h(t)-u(t)-u(t-\tau), \quad t\in (0,T],\\
\end{equation}
with the boundary condition
\[
u(t)=0, \quad t\in [-\tau, 0],
\]
where $\tau$ is taken as a fraction of the length of time interval $[0, 1]$. Now, two different cases for the forcing term $h(t)$ may be considered:\\
Case $(i)$:
\begin{equation}\label{ht1}
h(t)=
\begin{cases}
\frac{\Gamma(11)}{\Gamma(11-\nu)}t^{10-\nu}+t^{10},& t\in [0, \tau],\\
\frac{\Gamma(11)}{\Gamma(11-\nu)}t^{10-\nu}+t^{10}+(t-\tau)^{10},& t\in (\tau,1],
\end{cases}
\end{equation}
which corresponding exact solution is
\begin{equation}\label{aval}
u(t)=
\begin{cases}
0, & t\in [-\tau, 0],\\
t^{10},& t\in (0,1].
\end{cases}
\end{equation}
Case $(ii)$:
 \begin{equation}\label{ht2}
h(t)=
\begin{cases}
\sum_{j=0}^{\infty} \gamma_j\frac{\Gamma(\beta_j)}{\Gamma(\beta_j-\nu)}t^{\xi_j-\nu}+t^{\frac{13}{2}}\sin(\pi t^{\frac{4}{3}}), & t\in [0, \tau],\\
\sum_{j=0}^{\infty} \gamma_j\frac{\Gamma(\beta_j)}{\Gamma(\beta_j-\nu)}t^{\xi_j-\nu}+t^{\frac{13}{2}}\sin(\pi t^{\frac{4}{3}})
+(t-\tau)^{\frac{13}{2}}\sin(\pi (t-\tau)^{\frac{4}{3}}),& t\in (\tau,1],
\end{cases}
\end{equation}
where
\begin{align*}
  \gamma_j= & \frac{(-1)^j}{(2j+1)!}\pi^{2j+1}, \\
  \beta_j= & \frac{53+16j}{6}, \\
  \xi_j= & \frac{47+16j}{6}.
\end{align*}
The related exact solution is
\begin{equation}\label{dovom}
u(t)=\left\{
\begin{array}{lr}
0, \hspace{2.25cm} t\in [-\tau, 0],\\
t^{\frac{13}{2}}\sin(\pi t^{\frac{4}{3}}),\quad t\in (0,1].
\end{array}\right.
\end{equation}

Zayernouri et~al.~\cite{179} used Petrov-Galerkin spectral method to solve \eqref{Ex3}. They employed Reimann-Liouville fractional derivatives while we use the Caputo's fractional derivatives. As we know, these are related together by the following relation~\cite{Monje}
\[
\leftidx{_{\scs{R}}}D^{\nu}u(t)=D^{\nu}u(t)+\sum_{k=0}^{m-1}\frac{t^{k-\nu}}{\Gamma(k+1-\nu)}u^{(k)}(0^+),
\]
where $\leftidx{_{\scs{R}}}D^{\nu}$ stands for Reimann-Liouville fractional derivative. Since $u(0)=0$, both fractional derivatives are the same and consequently the resulted approximate solution are comparable.

The $\mathcal{L}^2$-Error of the Case $(i)$ and Case $(ii)$ for $\tau=0.5$ and different values of $N$ and $\nu=0.1$ are reported in Table~\ref{table:jad5} and Table~\ref{table:jad6}, respectively. The time responses of the two cases are plotted in Fig.~\ref{Fig5} and Fig.~\ref{Fig6}. As one can observe, the results are in remarkable agreement with the results of Zayernouri et~al.~\cite{179}. The results are obtained by employing shifted Chebyshev basis are more accurate than the others.
\begin{table}[htp]
\tabcolsep4.5pt
\centering
\caption{$\mathcal{L}^2$-Error of Example 1 (Case $(i)$) for different values of $N$ with $\tau=0.5$ and $\nu=0.1$.}
\begin{tabular}{ccc}
  \hline
  $N$&Shifted Legendre Basis&Shifted Chebyshev Basis\\
  \hline\hline
3&$0.871781$&$0.022220$\\
  5&$0.268627$&$0.002870$\\
  7&$0.020670$&$1.262008e{-4}$\\
  9&$1.863461e{-4}$&$7.422971e{-7}$\\
  11&$2.232967e{-8}$&$9.407349e{-9} $\\
    13&$1.954103e{-8}$&$1.039491e{-12}$\\
  15&$1.020807e{-8}$&$9.904553e{-14}$\\
  17&$6.124110e{-10}$&$9.425039e{-15}$\\
  19&$2.681213e{-12}$&$1.025904e{-15}$\\
  \hline
\end{tabular}
\label{table:jad5}
\end{table}
\begin{table}[hbp]
\tabcolsep4.5pt
\centering
\caption{$\mathcal{L}^2$-Error of Example 1 (Case $(ii)$) for different values of $N$ with $\tau=0.5$ and $\nu=0.1$.}
\begin{tabular}{ccc}
  \hline
  $N$&Shifted Legendre Basis&Shifted Chebyshev Basis\\
  \hline\hline
3&$0.343192$&$0.016382$\\
  5&$0.299374$&$0.002711$\\
  7&$0.006547$&$1.763871e{-4}$\\
  9&$0.003438$&$1.300984e{-5}$\\
  11&$1.724315e{-4}$&$7.690661e{-7}$\\
    13&$1.106626e{-6}$&$8.797766e{-9}$\\
  15&$4.971592e{-7}$&$2.055257e{-10}$\\
  17&$1.892239e{-8}$&$5.953249e{-12}$\\
  19&$ 9.786602e-{11}$&$1.270578e{-13}$\\
  \hline
\end{tabular}
\label{table:jad6}
\end{table}

\begin{figure}[htp]
\vspace{7.25cm}
\begin{picture}(0.,0.)
\put(-155.,-75.){\includegraphics{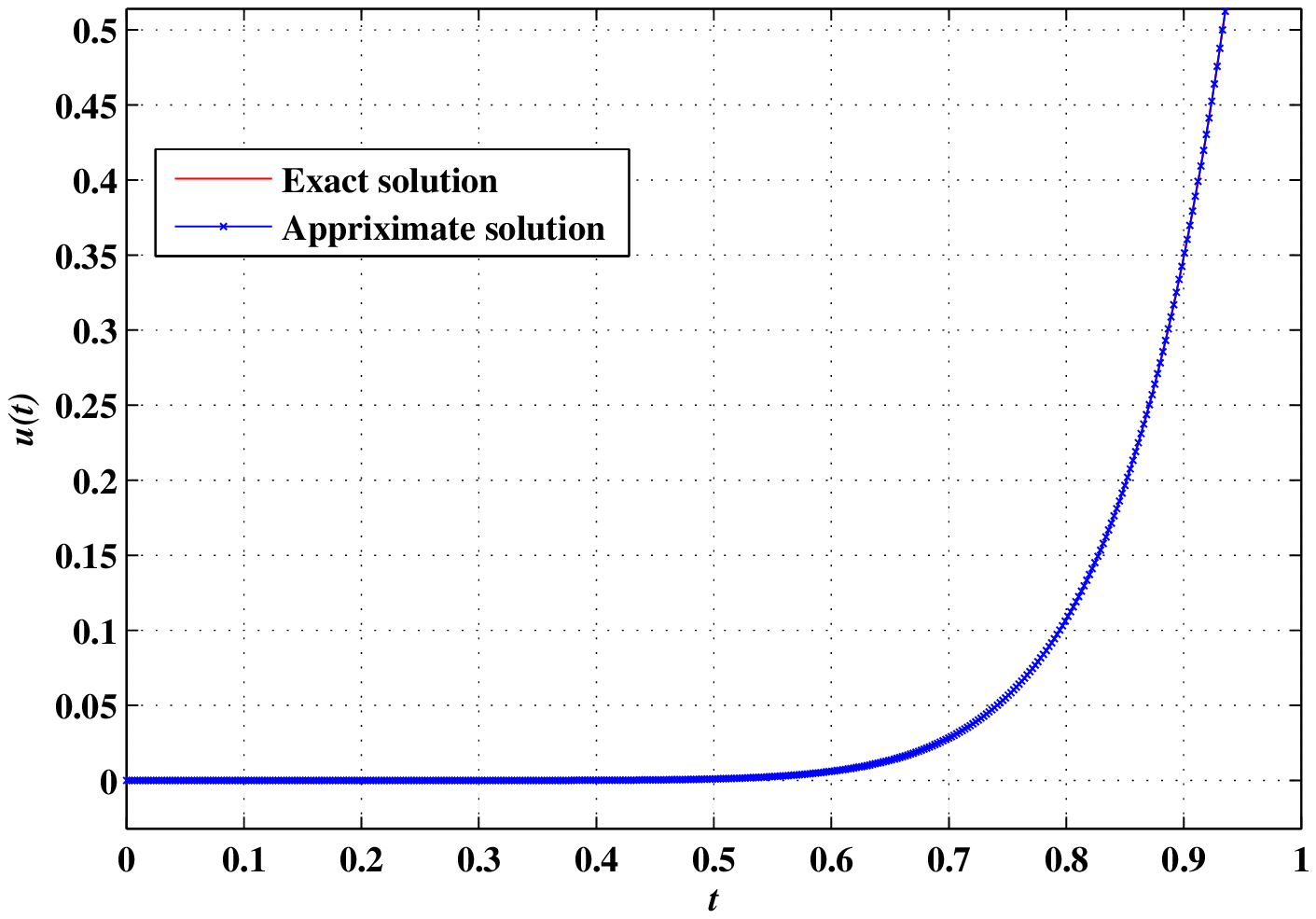}}
\end{picture}
\vspace{-2cm}
\caption{Time response of Example~1 (Case $(i)$) by the shifted Chebyshev basis with $N=15$ and $\tau=0.5$.}
 \label{Fig5}
\end{figure}
\begin{figure}[htp]
\vspace{7.25cm}
\begin{picture}(0.,0.)
\put(-155.,-75.){\includegraphics{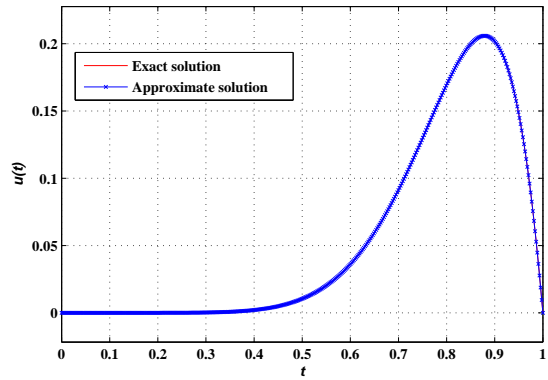}}
\end{picture}
\vspace{-2cm}
\caption{Time response of Example~1 (Case $(ii)$) by the shifted Chebyshev basis with $N=15$ and $\tau=0.5$.}
 \label{Fig6}
\end{figure}

\textbf{Example 2.} Consider the following~FDDE~\cite{179}
\begin{equation}\label{Ex4}
D^{\nu}u(t)=h(t)-A(t)u(t)-B(t)u(t-\tau), \quad t\in (0,1],\\
\end{equation}
with the initial condition
\[
u(t)=0, \quad t\in [-\tau, 0].
\]
Now, two cases are taken into consideration. Case $(i)$: Take $A(t)=B(t)=t^2-t^3$. The corresponding exact solution is given in~\eqref{aval}. Case $(ii)$: Put $A(t)=B(t)=\sin(\pi t)$, where the analytical solution is given in~\eqref{dovom}. The $\mathcal{L}^2$-Error of the Case $(i)$ and Case $(ii)$ for different values of $N$, $\tau=0.5$ and $\nu=0.1$ are provided in Table~\ref{table:jad7} and Table~\ref{table:jad8}, respectively. The $\mathcal{L}^2$-Error of the Case $(i)$ in the current work is at least of order $10^{-13}$ for $N\geq 11$, while it happened only when $N\geq 17$ in Ref.~\cite{179}. In the Case $(ii)$, the results are in remarkable agreement with the results of Zayernouri et~al.~\cite{179}. Again, the results are obtained by employing shifted Chebyshev basis are more accurate than the others.
\begin{table}[!hbp]
\tabcolsep4.5pt
\centering
\caption{$\mathcal{L}^2$-Error of Example 2 (Case $(i)$) for different values of $N$ with $\tau=0.5$ and $\nu=0.1$.}
\begin{tabular}{ccc}
  \hline
  $N$&Shifted Legendre Basis&Shifted Chebyshev Basis\\
  \hline\hline
3&$0.708611$&$0.655433$\\
  5&$0.113928$&$0.082273$\\
  7&$0.005367$&$0.003304$\\
  9&$3.324243e{-5}$&$1.697280e{-5}$\\
  11&$6.758311e{-13}$&$3.050396e{-14}$\\
    13&$5.424470e{-13}$&$3.024794e{-14}$\\
  15&$3.985062e{-13}$&$3.019384e{-14}$\\
  17&$2.606128e{-13}$&$3.009684e{-14}$\\
  19&$1.112418e{-13}$&$2.018803e{-14}$\\
  \hline
\end{tabular}
\label{table:jad7}
\end{table}

\begin{table}[!htp]
\tabcolsep4.5pt
\centering
\caption{$\mathcal{L}^2$-Error of Example 2 (Case $(ii)$) for different values of $N$ with $\tau=0.5$ and $\nu=0.1$.}
\begin{tabular}{ccc}
  \hline
  $N$&Shifted Legendre Basis&Shifted Chebyshev Basis\\
  \hline\hline
3&$0.447450$&$0.025522$\\
  5&$0.114932$&$0.003248$\\
  7&$0.005888$&$2.277593e-{4}$\\
  9&$6.020148e-{4}$&$1.489728e{-5}$\\
  11&$3.318842e{-5}$&$9.040363e{-7}$\\
    13&$3.284365e{-7}$&$1.067213e{-8}$\\
  15&$9.812455e{-9}$&$2.358068e{-10}$\\
  17&$3.035429e{-10}$&$6.944012e{-12}$\\
  19&$2.880012e{-11}$&$1.692653e{-13}$\\
  \hline
\end{tabular}
\label{table:jad8}
\end{table}

\textbf{Example 3.} Consider Houseflies model as following~\cite{6}
\begin{equation}\label{Ex1}
D^{\nu}u(t)=-du(t)+cu(t-\tau)(k-czu(t-\tau)),\quad t>0,\\
\end{equation}
with the initial condition
\[
u(t)=160,\quad -\tau\leq t\leq 0.
\]
By taking $c= 1.81$, $k=0.5107$, $d=0.147$ and $z=0.000226$, numerical results of the shifted Chebyshev basis for different values of $\nu$, $\tau=3$ and $\tau=5$ are presented in Table~\ref{table:jad1} and Table~\ref{table:jad2}, respectively. Table~\ref{table:jad1L} and Table~\ref{table:jad2L} describe the numerical results of the shifted Legendre basis with the same parameters. The approximate solutions are sketched in Fig.~\ref{Fig1} and Fig.~\ref{Fig2}.  Comparison between the second and third columns of the Table~\ref{table:jad1} and Table~\ref{table:jad2} (Table~\ref{table:jad1L} and Table~\ref{table:jad2L}) reveals that the maximum absolute error (MAE) is $2\times 10^{-6}$ while the MAE reported in Ref.~\cite{6} which employed the finite difference method was of order $10^{-5}$ using $N=100$. Additionally, phase-space solution for $\nu=0.5$ and $\tau=5$ is plotted in Fig.~\ref{Fig00}. Moreover, log plot of MAE for different values of $N$, $\tau=3$ and $\tau=5$ are plotted in Fig.~\ref{NF1} and Fig.~\ref{NF2}, respectively.
\begin{table}[!htp]
\tabcolsep4.5pt
\centering
\caption{Numerical results of Example 3 by the shifted Chebyshev basis with $N=15$ and $\tau=3$.}
\begin{tabular}{cccccc}
  \hline
  $t$&exact $\nu=1$&$\nu=1$&$\nu=0.9$&$\nu=0.75$&$\nu=0.5$\\
  \hline\hline
  0.0&$160.000000$&$160.000000$&$160.000000$&$160.000000$&$160.000000$\\
  0.75&$234.865602$&$234.865602$&$239.333361$&$245.219089$&$252.133912$\\
    2.25&$361.967021$&$361.967021$&$352.369942$&$336.687494$&$308.795525$\\
    3.75&$481.825305$&$481.825304$&$472.760850$&$459.813354$&$439.073221$\\
    5.25&$670.725206$&$670.725204$&$650.373599$&$614.508045$&$543.799282$\\
    6.0&$776.578086$&$776.578085$&$742.142667$&$685.419046$&$583.383909$\\
  \hline
\end{tabular}
\label{table:jad1}
\end{table}

\begin{table}[!htp]
\tabcolsep4.5pt
\centering
\caption{Numerical results of Example 3 by the shifted Chebyshev basis with $N=15$ and $\tau=5$.}
\begin{tabular}{cccccc}
  \hline
  $t$&exact $\nu=1$&$\nu=1$&$\nu=0.9$&$\nu=0.75$&$\nu=0.5$\\
  \hline\hline
  0.0&$160.000000$&$160.000000$&$160.000000$&$160.000000$&$160.000000$\\
  1.25&$280.382021$&$280.382021$&$281.116516$&$280.559015$&$275.428893$\\
    3.75&$463.917311$&$463.917311$&$437.587058$&$399.856049$&$342.930615$\\
    6.25&$636.682069$&$636.682068$&$609.967593$&$571.739341$&$510.393022$\\
    8.75&$950.311527$&$950.311525$&$890.874639$&$798.685328$&$645.755656$\\
    10.0&$1107.006007$&$1107.006006$&$1022.495184$&$895.594663$&$695.358908$\\
  \hline
\end{tabular}
\label{table:jad2}
\end{table}
\begin{table}[!htp]
\tabcolsep4.5pt
\centering
\caption{Numerical results of Example 3 by the shifted Legendre basis with $N=15$ and $\tau=3$.}
\begin{tabular}{cccccc}
  \hline
  $t$&exact $\nu=1$&$\nu=1$&$\nu=0.9$&$\nu=0.75$&$\nu=0.5$\\
  \hline\hline
  0.0&$160.000000$&$160.000000$&$160.000000$&$160.000000$&$160.000000$\\
  0.75&$234.865602$&$234.865601$&$239.333372$&$245.219085$&$252.133922$\\
    2.25&$361.967021$&$361.967023$&$352.369941$&$336.687497$&$308.795514$\\
    3.75&$481.825305$&$481.825304$&$472.760849$&$459.813353$&$439.073219$\\
    5.25&$670.725206$&$670.725202$&$650.373596$&$614.508046$&$543.799279$\\
    6.0&$776.578086$&$776.578087$&$742.142661$&$685.419042$&$583.383904$\\
  \hline
\end{tabular}
\label{table:jad1L}
\end{table}
\begin{table}[!htp]
\tabcolsep4.5pt
\centering
\caption{Numerical results of Example 3 by the shifted Legendre basis with $N=15$ and $\tau=5$.}
\begin{tabular}{cccccc}
  \hline
  $t$&exact $\nu=1$&$\nu=1$&$\nu=0.9$&$\nu=0.75$&$\nu=0.5$\\
  \hline\hline
  0.0&$160.000000$&$160.000000$&$160.000000$&$160.000000$&$160.000000$\\
  1.25&$280.382021$&$280.382023$&$281.116518$&$280.559021$&$275.428895$\\
    3.75&$463.917311$&$463.917314$&$437.587055$&$399.856060$&$342.930617$\\
    6.25&$636.682069$&$636.682069$&$609.967581$&$571.739351$&$510.393025$\\
    8.75&$950.311527$&$950.311526$&$890.874641$&$798.685332$&$645.755659$\\
    10.0&$1107.006007$&$1107.006008$&$1022.495191$&$895.594671$&$695.358910$\\
  \hline
\end{tabular}
\label{table:jad2L}
\end{table}
\begin{figure}[!htp]
\vspace{7.cm}
\begin{picture}(0.,0.)
\put(-155.,-75.){\includegraphics{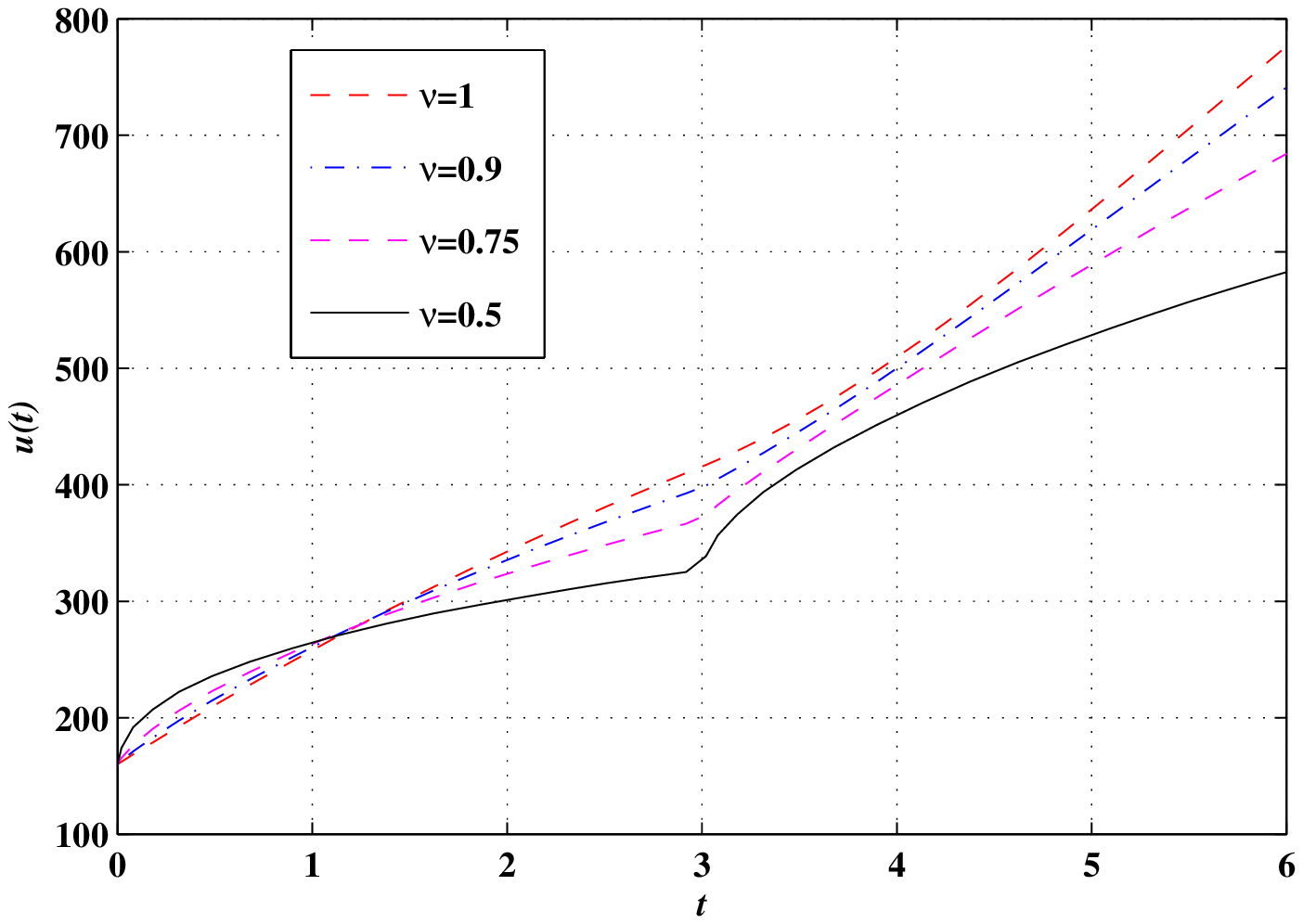}}
\end{picture}
\vspace{-2cm}
\caption{Approximate solution of Example~3 using the shifted Chebyshev basis with $N=15$ and $\tau=3$.}
 \label{Fig1}
\end{figure}
\begin{figure}[hbp]
\vspace{7.25cm}
\begin{picture}(0.,0.)
\put(-155.,-75.){\includegraphics{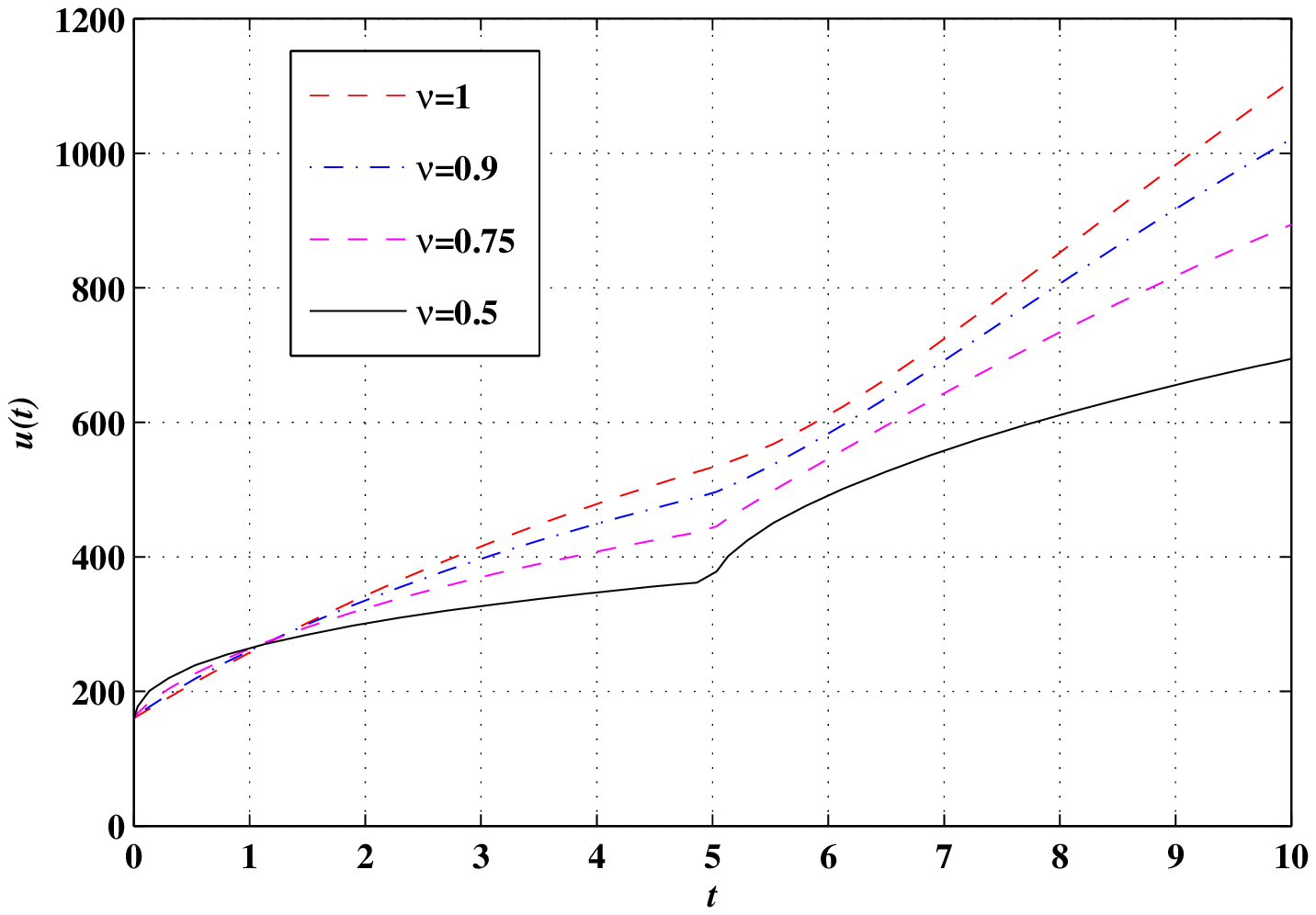}}
\end{picture}
\vspace{-2cm}
\caption{Approximate solution of Example~3 using the shifted Chebyshev basis with $N=15$ and $\tau=5$.}
 \label{Fig2}
\end{figure}
\begin{figure}[htp]
\vspace{7.25cm}
\begin{picture}(0.,0.)
\put(-155.,-75.){\includegraphics{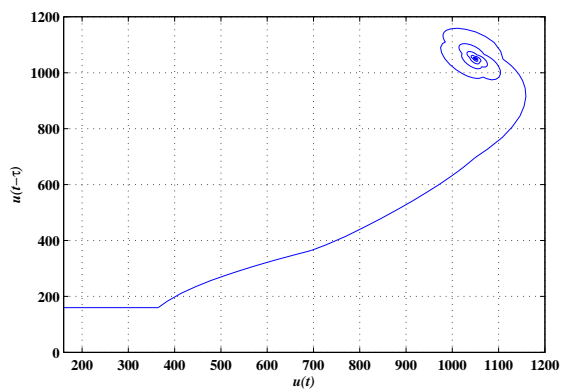}}
\end{picture}
\vspace{-2cm}
\caption{Phase-space solution of Example~3 using the shifted Chebyshev basis with $N=15$, $\nu=0.5$ and $\tau=5$.}
 \label{Fig00}
\end{figure}

\begin{figure}[htp]
\vspace{7.25cm}
\begin{picture}(0.,0.)
\put(-155.,-75.){\includegraphics{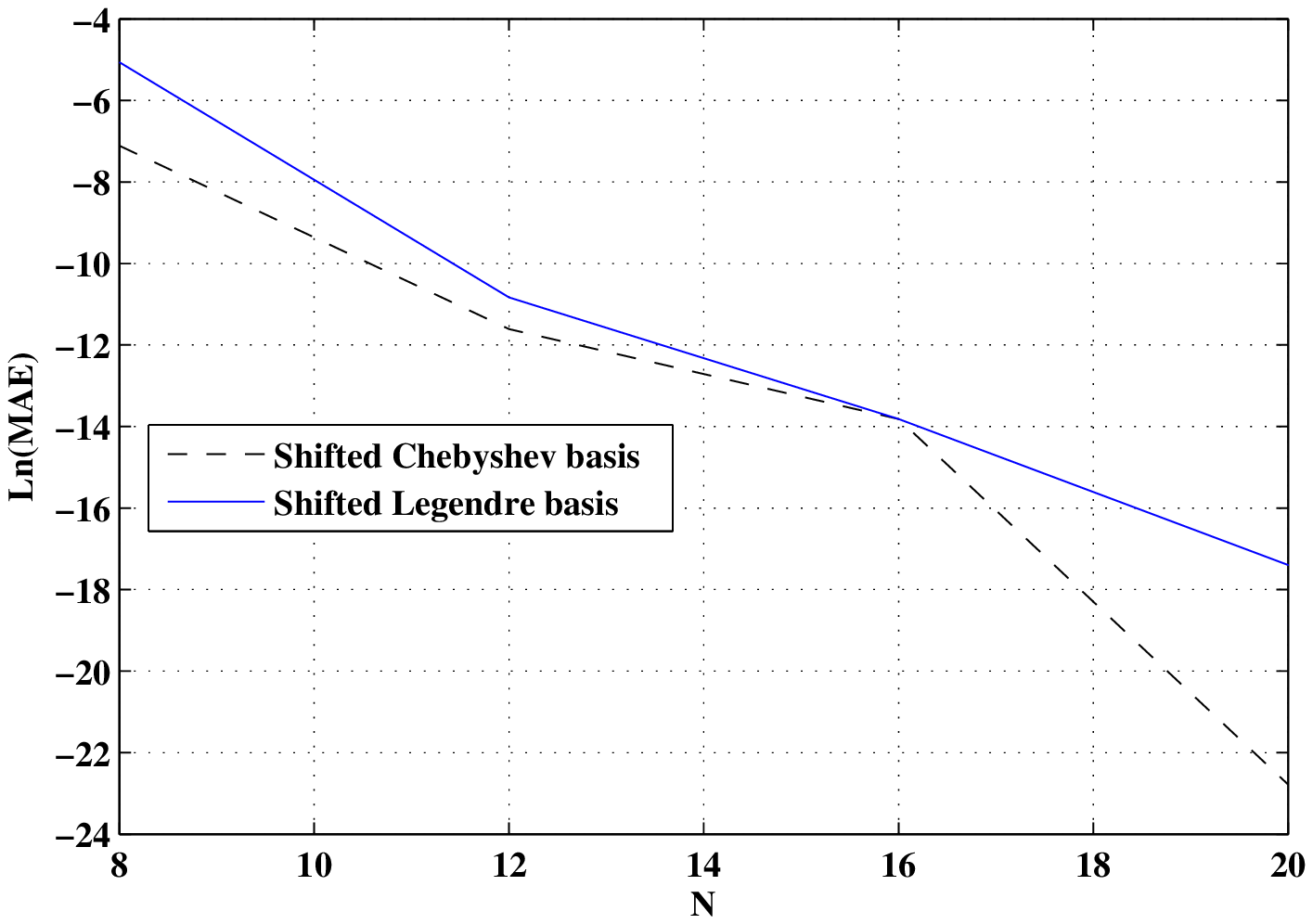}}
\end{picture}
\vspace{-2cm}
\caption{Log plot of MAE of Example~3 with $\nu=1$ and $\tau=3$.}
 \label{NF1}
\end{figure}

\begin{figure}[htp]
\vspace{7.25cm}
\begin{picture}(0.,0.)
\put(-155.,-75.){\includegraphics{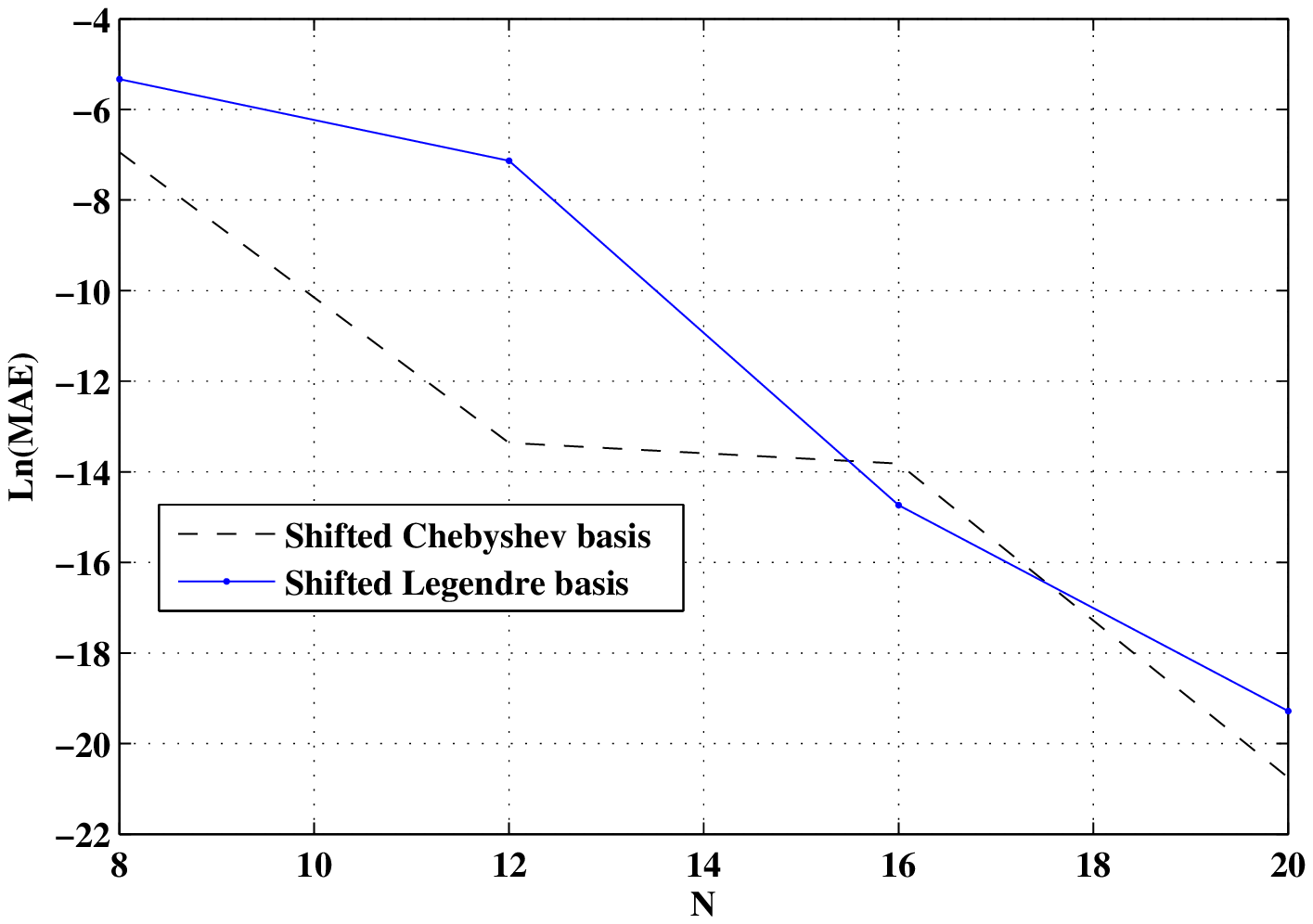}}
\end{picture}
\vspace{-2cm}
\caption{Log plot of MAE of Example~3 with $\nu=1$ and $\tau=5$.}
 \label{NF2}
\end{figure}

\textbf{Example 4.} The following model example concerns with the effect of noise on light which is introduced by Pieroux~\cite{6}
\begin{equation}\label{Ex2}
D^{\nu}u(t)=-\frac{1}{\epsilon}u(t)+\frac{1}{\epsilon}u(t)u(t-\tau), \quad t>0,\\
\end{equation}
with the initial condition
\[
u(t)=0.9,\quad -\tau\leq t\leq 0.
\]
The obtained results of the shifted Chebyshev basis for various values of $\nu$ and $\tau$ with $\epsilon=0.1$ are reported in Table~\ref{table:jad3} and Table~\ref{table:jad4}. Also, numerical results of the shifted Legendre basis are given in Table~\ref{table:jad3L} and Table~\ref{table:jad4L}. In this model example, the MAEs related to the current works are order of $10^{-6}$ using $N=15$, while the MAE reported in Ref.~\cite{6} which employed the finite difference method achieved to this order of accuracy using $N=100$ nodes. Fig.~\ref{Fig3} and Fig.~\ref{Fig4} show the approximate solutions, whereas Fig.~\ref{Fig01} demonstrates the phase-space solution for $\nu=0.5$ and $\tau=5$. Also, log plot of MAE for $\tau=1$ and $\tau=3$ are plotted in Fig.~\ref{NF3} and Fig.~\ref{NF4}, respectively.
\begin{table}[!htp]
\tabcolsep4.5pt
\centering
\caption{Numerical results of Example 4 by the shifted Chebyshev basis with $N=15$ and $\tau=1$.}
\begin{tabular}{cccccc}
  \hline
  $t$&exact $\nu=1$&$\nu=1$&$\nu=0.9$&$\nu=0.75$&$\nu=0.5$\\
  \hline\hline
  0.0&$0.900000$&$0.900000$&$0.900000$&$0.900000$&$0.900000$\\
  0.25&$0.700921$&$0.700921$&$0.670849$&$0.625076$&$0.555064$\\
    0.75&$0.425130$&$0.425130$&$0.417343$&$0.413280$&$0.419605$\\
    1.25&$0.198977$&$0.198979$&$0.172599$&$0.135910$&$0.097493$\\
    1.75&$0.021138$&$0.021138$&$0.024889$&$0.032106$&$0.044666$\\
    2.0&$0.004444$&$0.004444$&$0.011581$&$0.021070$&$0.036340$\\
  \hline
\end{tabular}
\label{table:jad3}
\end{table}
\begin{table}[!hbp]
\tabcolsep4.5pt
\centering
\caption{Numerical results of Example 4 by the shifted Chebyshev basis with $N=15$ and $\tau=3$.}
\begin{tabular}{cccccc}
  \hline
  $t$&exact $\nu=1$&$\nu=1$&$\nu=0.9$&$\nu=0.75$&$\nu=0.5$\\
  \hline\hline
  0.0&$0.900000$&$0.900000$&$0.900000$&$0.900000$&$0.900000$\\
  0.75&$0.425130$&$0.425130$&$0.417917$&$0.414396$&$0.421697$\\
    2.25&$0.094859$&$0.094859$&$0.137716$&$0.197791$&$0.288194$\\
    3.75&$0.002861$&$0.002859$&$0.006645$&$0.014154$&$0.030316$\\
    5.25&$0.000000$&$0.000000$&$0.000637$&$0.003305$&$0.013782$\\
    6.0&$0.000000$&$0.000000$&$0.000429$&$0.002413$&$0.011416$\\
  \hline
\end{tabular}
\label{table:jad4}
\end{table}
\begin{table}[!htp]
\tabcolsep4.5pt
\centering
\caption{Numerical results of Example 4 by the shifted Legendre basis with $N=15$ and $\tau=1$.}
\begin{tabular}{cccccc}
  \hline
  $t$&exact $\nu=1$&$\nu=1$&$\nu=0.9$&$\nu=0.75$&$\nu=0.5$\\
  \hline\hline
  0.0&$0.900000$&$0.900000$&$0.900000$&$0.900000$&$0.900000$\\
  0.25&$0.700921$&$0.700923$&$0.670842$&$0.625061$&$0.555051$\\
    0.75&$0.425130$&$0.425132$&$0.417353$&$0.413275$&$0.419613$\\
    1.25&$0.198977$&$0.198976$&$0.172538$&$0.135923$&$0.097481$\\
    1.75&$0.021138$&$0.021134$&$0.024892$&$0.032112$&$0.044652$\\
    2.0&$0.004444$&$0.004446$&$0.011572$&$0.021083$&$0.036361$\\
  \hline
\end{tabular}
\label{table:jad3L}
\end{table}
\begin{table}[!htp]
\tabcolsep4.5pt
\centering
\caption{Numerical results of Example 4 by the shifted Legendre basis with $N=15$ and $\tau=3$.}
\begin{tabular}{cccccc}
  \hline
  $t$&exact $\nu=1$&$\nu=1$&$\nu=0.9$&$\nu=0.75$&$\nu=0.5$\\
  \hline\hline
  0.0&$0.900000$&$0.900000$&$0.900000$&$0.900000$&$0.900000$\\
  0.75&$0.425130$&$0.425130$&$0.417921$&$0.414410$&$0.421730$\\
    2.25&$0.094859$&$0.094855$&$0.137725$&$0.197820$&$0.288230$\\
    3.75&$0.002861$&$0.002854$&$0.006653$&$0.014168$&$0.030334$\\
    5.25&$0.000000$&$0.000001$&$0.000648$&$0.003319$&$0.013797$\\
    6.0&$0.000000$&$0.000003$&$0.000437$&$0.002443$&$0.011443$\\
  \hline
\end{tabular}
\label{table:jad4L}
\end{table}
\begin{figure}[htp]
\vspace{7.25cm}
\begin{picture}(0.,0.)
\put(-155.,-75.){\includegraphics{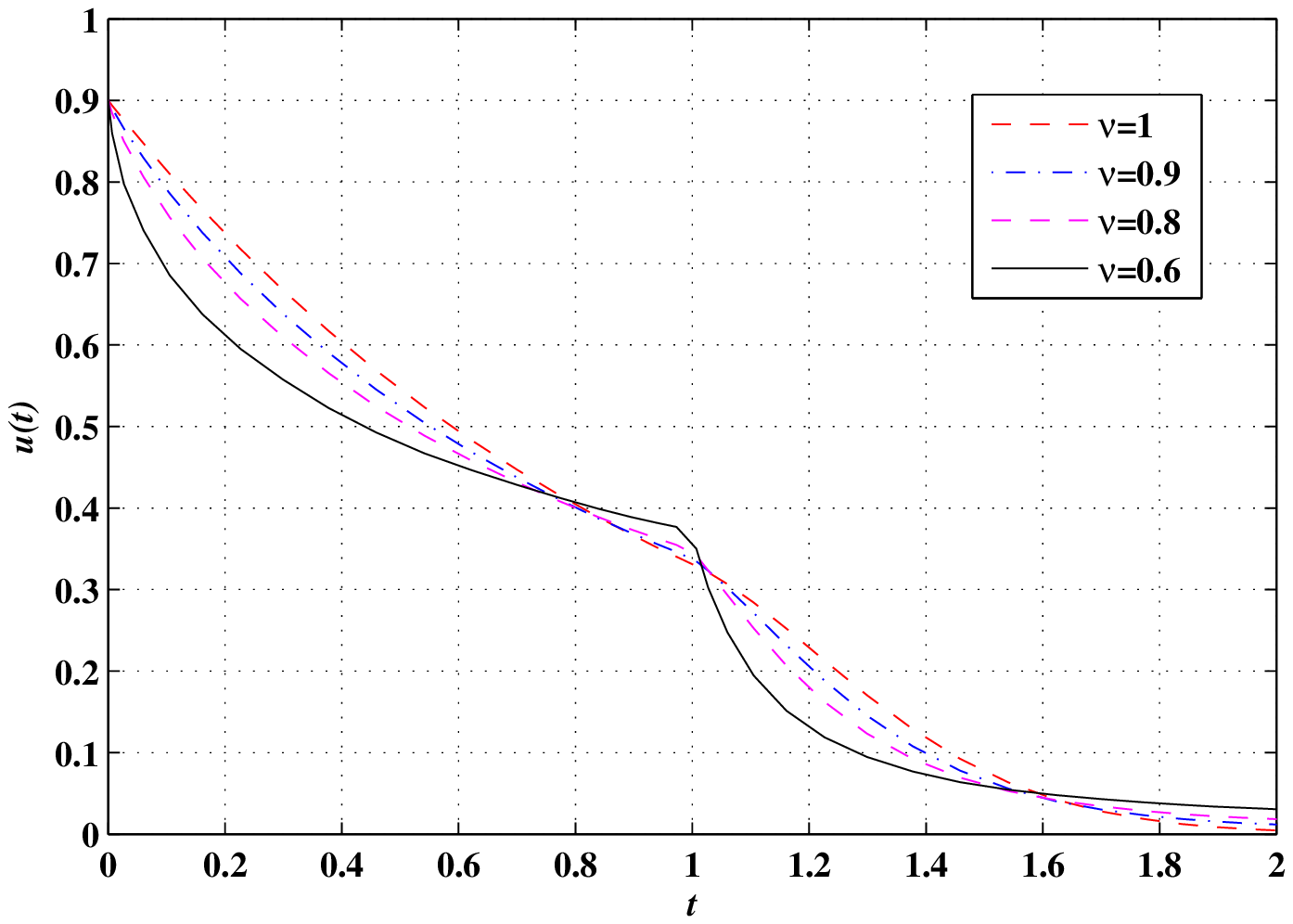}}
\end{picture}
\vspace{-2cm}
\caption{Approximate solution of Example~4 by the shifted Legendre basis with $N=15$ and $\tau=1$.}
 \label{Fig3}
\end{figure}
\begin{figure}[!htp]
\vspace{8.25cm}
\begin{picture}(0.,0.)
\put(-155.,-75.){\includegraphics{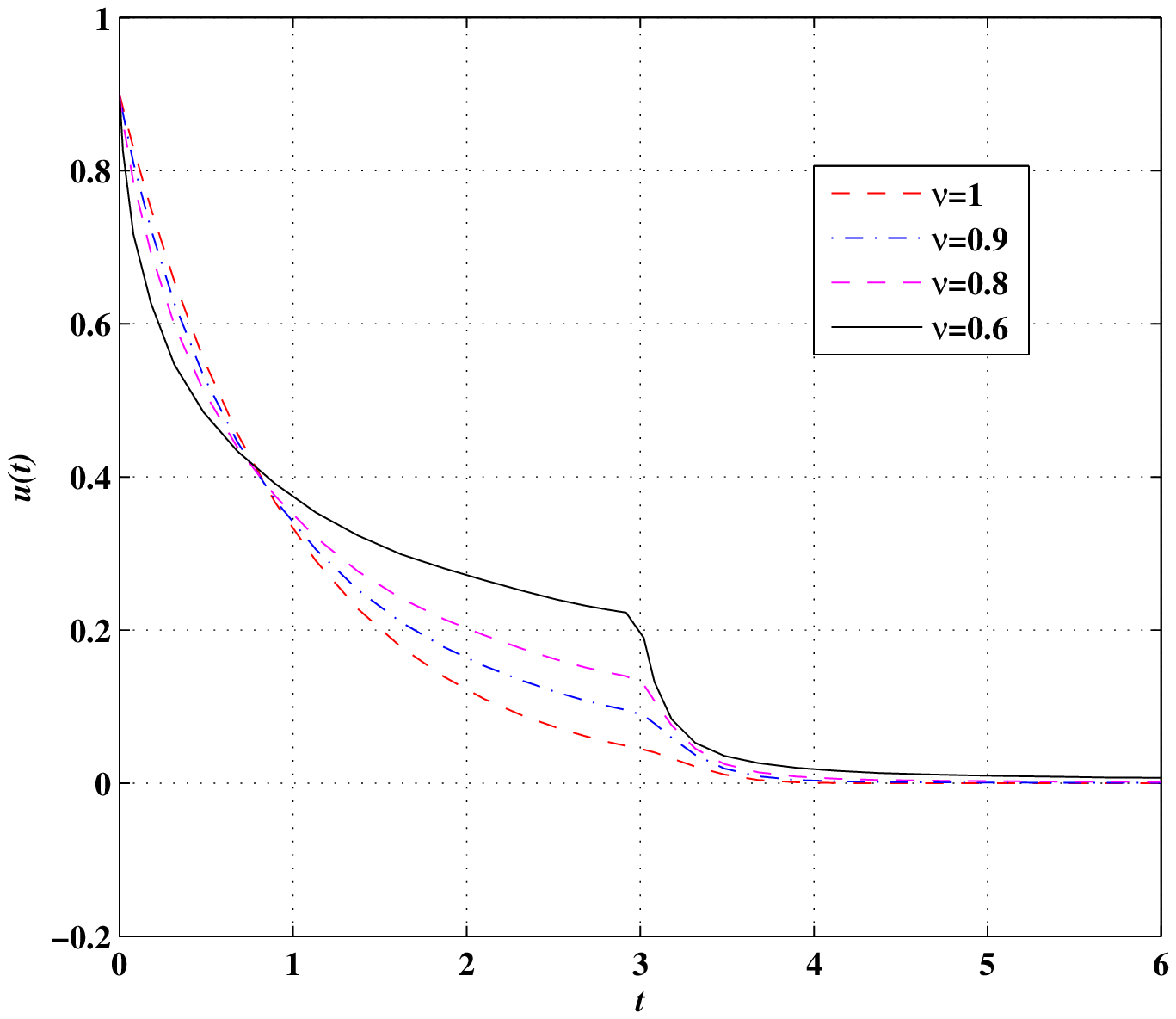}}
\end{picture}
\vspace{-1.5cm}
\caption{Approximate solution of Example~4 by the shifted Legendre basis with $N=15$ and $\tau=3$.}
 \label{Fig4}
\end{figure}
\begin{figure}[!htp]
\vspace{7.25cm}
\begin{picture}(0.,0.)
\put(-155.,-75.){\includegraphics{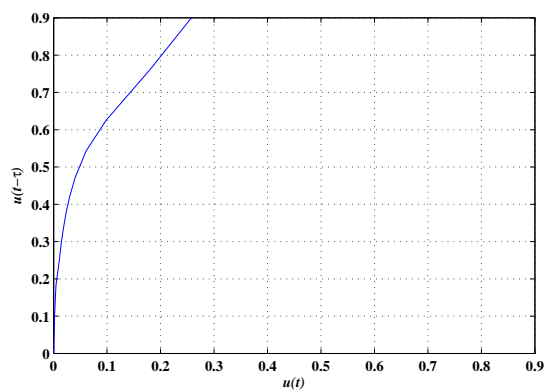}}
\end{picture}
\vspace{-2.0cm}
\caption{Phase-space solution of Example~4 by the shifted Legendre basis with $N=15$, $\nu=0.5$ and $\tau=3$.}
 \label{Fig01}
\end{figure}
\begin{figure}[htp]
\vspace{7.25cm}
\begin{picture}(0.,0.)
\put(-155.,-75.){\includegraphics{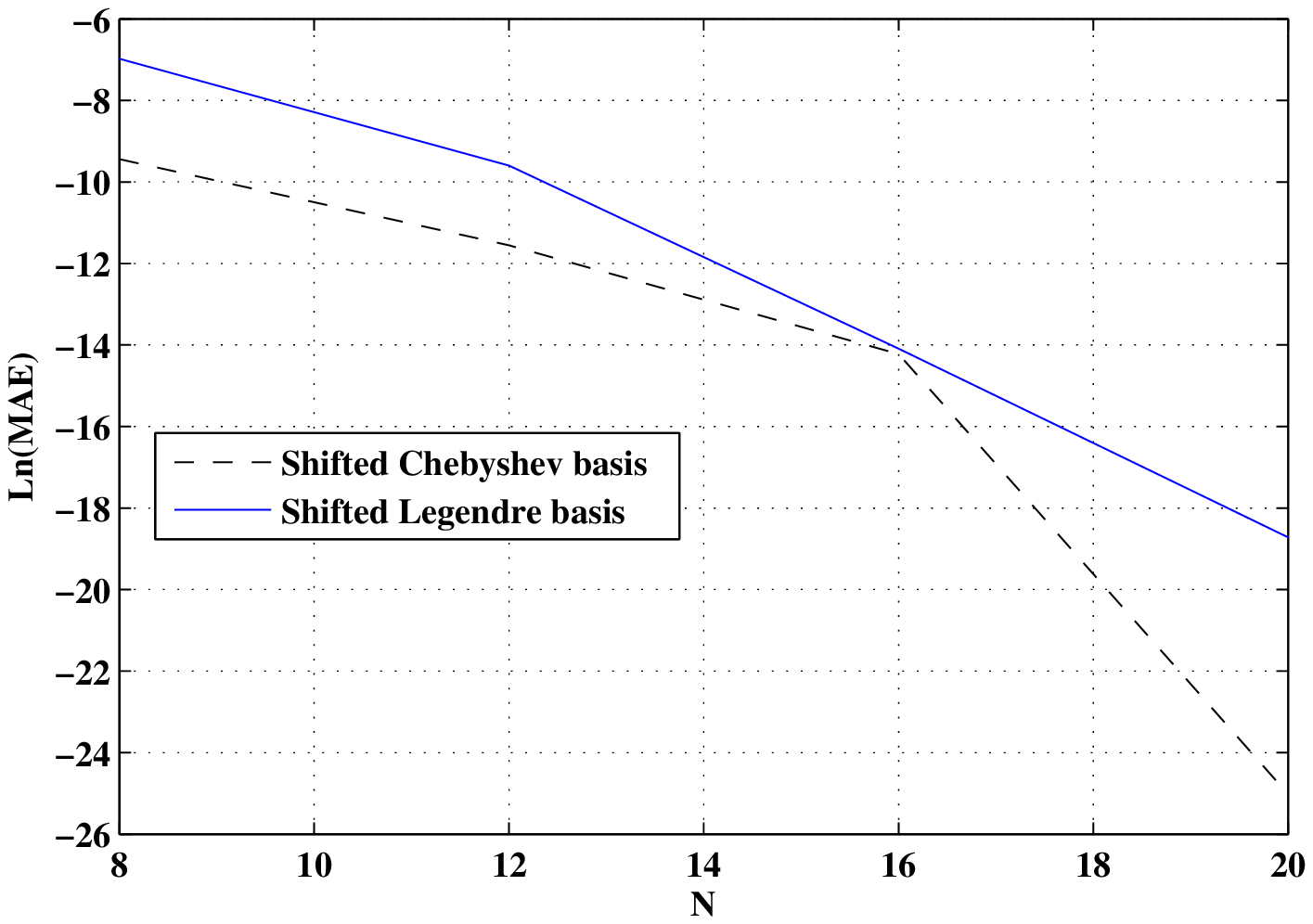}}
\end{picture}
\vspace{-2cm}
\caption{Log plot of MAE of Example~4 with $\nu=1$ and $\tau=1$.}
 \label{NF3}
 \end{figure}

 \begin{figure}[htp]
\vspace{7.25cm}
\begin{picture}(0.,0.)
\put(-155.,-75.){\includegraphics{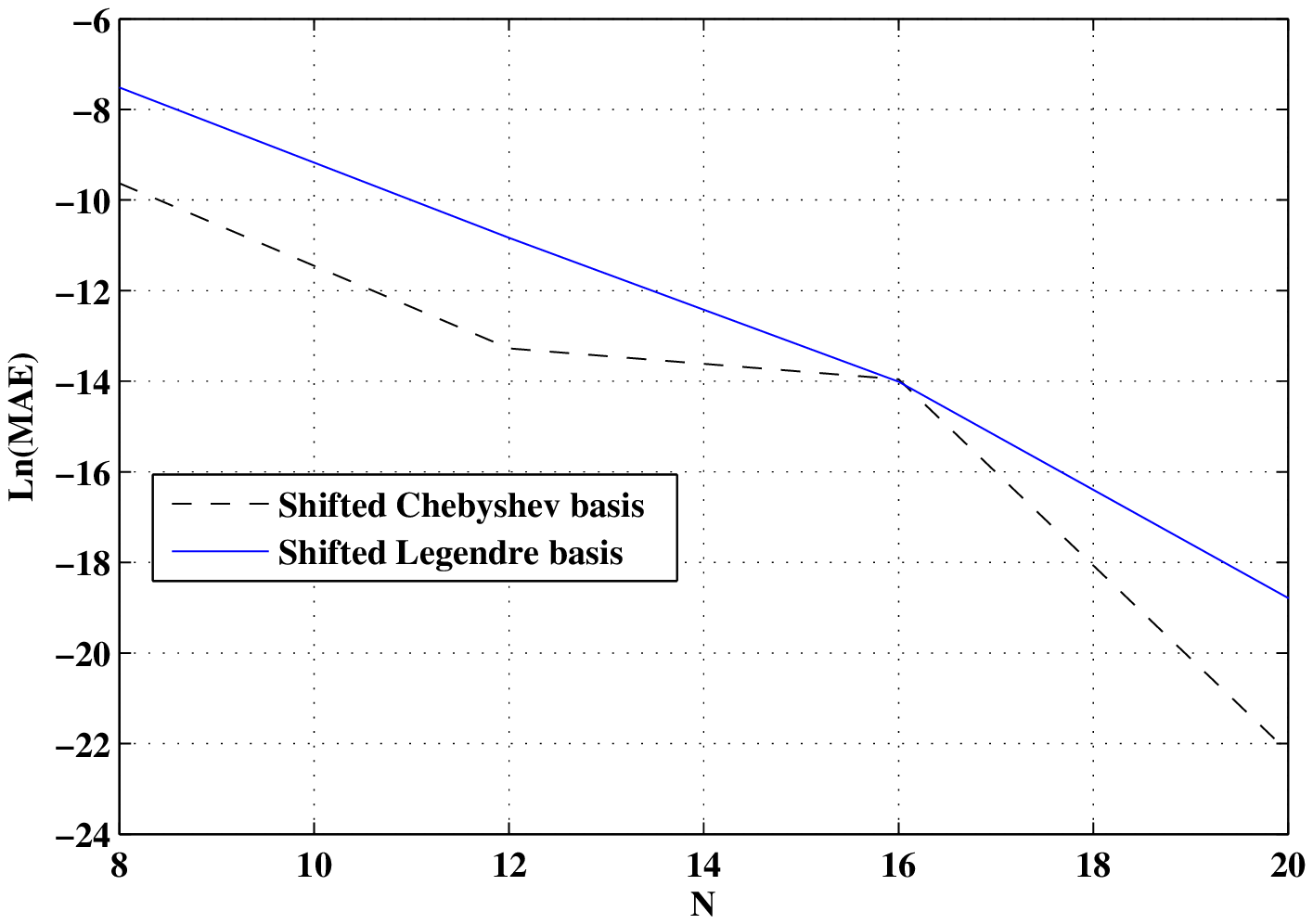}}
\end{picture}
\vspace{-2cm}
\caption{Log plot of MAE of Example~4 with $\nu=1$ and $\tau=3$.}
 \label{NF4}
\end{figure}

\textbf{Example 5.} As a final model example, consider the following FDDE which is introduced by Huang and Williams~\cite{180}
\begin{equation}\label{Ex223}
D^{\nu}u(t)+\delta D^{\nu_1}u(t)=-u(t)+\frac{\mu q^2}{u^3(t)}(u(t)-\gamma u(t-\tau)), \quad t\in (0,1),\\
\end{equation}
equipped with the conditions
\[
u(t)=1, \quad t\in [-\tau, 0], \quad u(1)=3,
\]
where $\delta= 0.3$, $\mu = 1$, $q = 0.4$ and $\gamma =0.2$. Computational results of the shifted Legendre basis with $N=15$ and $\tau=5$ is reported in Table~\ref{table:jad9L} and Table~\ref{table:jad9} demonstrates the results using the shifted Chebyshev basis. A comparison between the second and third columns of Table~\ref{table:jad9} and Table~\ref{table:jad9L} show that the present work is in remarkable agreement with the function $\text{bvp4c}$ of the Matlab software. However, the MAE of the finite difference method at $t=1$ is of order $10^{-2}$~~\cite{180}, but in both presented schemes we get the exact results. The graph of the numerical solutions of \eqref{Ex223} for different values of $\nu$ and $\nu_1$ is sketched in Fig.~\ref{Fig7}. Furthermore, Fig.~\ref{Fig8} confirms the agreement between the present work and the function $\text{bvp4c}$ of the Matlab.
\begin{table}[htp]
\tabcolsep4.5pt
\centering
\caption{Numerical results of Example~5 by the shifted Chebyshev basis with $N=15$ and $\tau=5$.}
\begin{tabular}{c||c|c|ccc}
  \hline
   $t$& \multicolumn{2}{c|}{$\nu=2\&\nu_1=1$}&$\nu=1.5\& \nu_1=0.5$&$\nu=1.75\&\nu_1=0.75$&$\nu=1.95\& \nu_1=0.95$\\
  \hline
  &Current work&bvp4c&\multicolumn{3}{c}{Current work}\\
            \cline{2-6}
0&$1.000000$&$1.000000$&$1.000000$&$1.000000$&$1.000000$\\
  0.25&$1.758262$&$1.758281$&$1.830901$&$1.800878$&$1.767155$\\
  0.75&$2.777401$&$2.777411$&$2.802636$&$2.800559$&$2.783206$\\
  1&$3.000000$&$3.000000$&$3.000000$&$3.000000$&$3.000000$\\
  \hline
\end{tabular}
\label{table:jad9}
\end{table}
\begin{table}[!htp]
\tabcolsep4.5pt
\centering
\caption{Numerical results of Example~5 by the shifted Legendre basis with $N=15$ and $\tau=5$.}
\begin{tabular}{c||c|c|ccc}
  \hline
   $t$& \multicolumn{2}{c|}{$\nu=2\&\nu_1=1$}&$\nu=1.5\& \nu_1=0.5$&$\nu=1.75\&\nu_1=0.75$&$\nu=1.95\& \nu_1=0.95$\\
  \hline
  &Current work&bvp4c&\multicolumn{3}{c}{Current work}\\
            \cline{2-6}
0&$1.000000$&$1.000000$&$1.000000$&$1.000000$&$1.000000$\\
  0.25&$1.758275$&$1.758281$&$1.830918$&$1.800884$&$1.767174$\\
  0.75&$2.777414$&$2.777411$&$2.802647$&$2.800557$&$2.783224$\\
  1&$3.000000$&$3.000000$&$3.000000$&$3.000000$&$3.000000$\\
  \hline
\end{tabular}
\label{table:jad9L}
\end{table}
\begin{figure}[htp]
\vspace{7.25cm}
\begin{picture}(0.,0.)
\put(-155.,-75.){\includegraphics{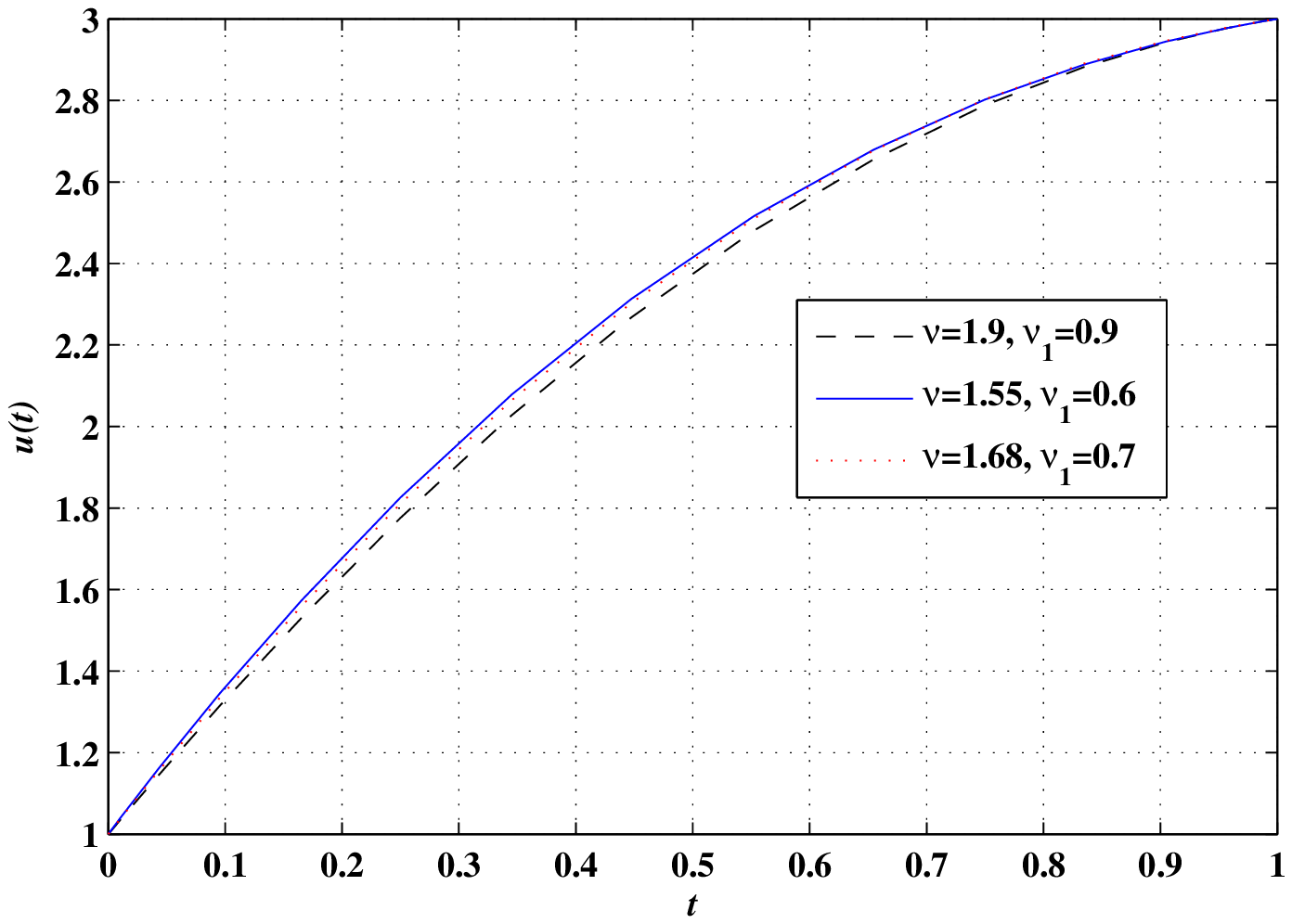}}
\end{picture}
\vspace{-2cm}
\caption{Approximate solution of Example~5 by the shifted Legendre basis with $N=15$ and $\tau=5$.}
 \label{Fig7}
\end{figure}
\begin{figure}[!htp]
\vspace{7.25cm}
\begin{picture}(0.,0.)
\put(-155.,-75.){\includegraphics{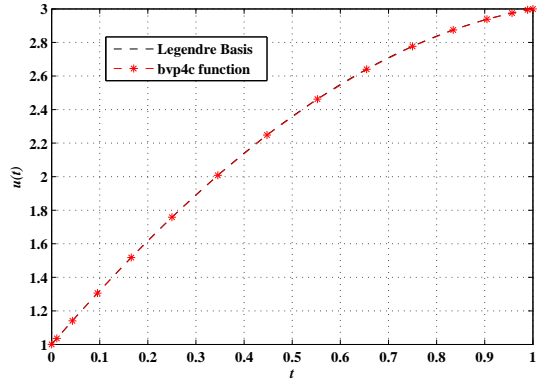}}
\end{picture}
\vspace{-2cm}
\caption{Comparison between the function bvp4c and shifted Legendre basis solutions of Example~5 with $N=15$ and $\tau=5$, where $\nu=2$ and $\nu_1=1$.}
 \label{Fig8}
\end{figure}

The CPU time of the above examples is reported in Table~\ref{table:cpu}. As we see from the table, the CPU time of the shifted Chebyshev basis is less than Legendre one.

\begin{table}[!htbp]
\vspace{.5cm}
\centering
\caption{CPU time (s) for the proposed schemes with $N=15$.}
\begin{tabular}{|c|c|c||c|c||c|}
  \hline
  Example&Shifted Legendre basis&Shifted Chebyshev basis&$\tau$&$\nu$&$\nu_1$\\
  \hline\hline
  1(case i) &$0.765$&$0.621$&$0.5$&$0.1$&$-$\\
  1(case ii)&$1.852$&$1.631$&$0.5$&$0.1$&$-$\\
  2(case i)&$0.462$&$0.283$&$0.5$&$0.1$&$-$\\
  2(case ii)&$0.542$&$0.321$&$0.5$&$0.1$&$-$\\
  3&$1.024$&$0.826$&$3$&$0.9$&$-$\\
  4&$4.657$&$3.245$&$1$&$0.9$&$-$\\
  5&$2.241$&$2.798$&$5$&$1.5$&$0.5$\\
  \hline
\end{tabular}
\label{table:cpu}
\end{table}


\section{Conclusion}
In this article, a new formula for fractional derivatives of shifted Legendre polynomials is derived. All the fractional derivatives are considered in the Caputo sense. By using the formula and the formula based on shifted Chebyshev polynomials for fractional derivatives~\cite{Mousa}, numerical schemes for solving nonlinear FDDEs are proposed. The proposed schemes exploit the method of steps and shifted Legendre~(Chebyshev) basis to generate an approximate solution. A mathematical analysis shows that the proposed schemes have an exponential rate of convergence. Moreover, practical examples are taken to demonstrate the effectiveness of the obtained results. MAE reveals that the approximate solution has acceptable conformity with the available literature. Further development of the proposed schemes should be concentrated on solving nonlinear fractional delay differential problems with more than one delay. It would also be interesting to extend an approximate solution in which a discontinuous nonlinear $f$ is considered.

\appendix
\section{Proof of relation~(31)}\label{Ap1}
Integrating by parts of the left-hand side of~\eqref{si3} yields
\begin{align}
\frac{1}{\Gamma(m-\nu)}\int_{(i-1)\tau}^{t}\frac{(t-s)^{n-1}}{(n-1)!}\int_{(i-1)\tau}^{s}(s-s_1)^{m-\nu-1}\leftidx{_{\scs{(i)}}}e_{\scs N}^{(m)}(s_1)\mathrm{d}s_1\mathrm{d}s\hspace{6cm}\nonumber\\
=\frac{1}{\Gamma(m-\nu)(n-1)!(m-\nu)}\int_{(i-1)\tau}^{s}(t-s)^{n-1}(s-s_1)^{m-\nu}\leftidx{_{\scs{(i)}}}e_{\scs N}^{(m)}(s_1)|_{s=(i-1)\tau}^{s=t}\mathrm{d}s_1\hspace{4cm}\nonumber\\
+\frac{1}{\Gamma(m-\nu)(n-2)!(m-\nu)}\int_{(i-1)\tau}^{t}\int_{(i-1)\tau}^{s}(t-s)^{n-2}(s-s_1)^{m-\nu}\leftidx{_{\scs{(i)}}}e_{\scs N}^{(m)}(s_1)\mathrm{d}s_1\mathrm{d}s.\hspace{4cm}\nonumber
\end{align}
As we know, one can replace $\int_{(i-1)\tau}^{t}\int_{(i-1)\tau}^{s}F\mathrm{d}s_1\mathrm{d}s$ by $\int_{(i-1)\tau}^{t}\int_{s_1}^{t}F\mathrm{d}s\mathrm{d}s_1$. Therefore, the first term of the right hand side of the above relation becomes zero when we substitute the limits of integration. Thus, we have
\begin{align}
\frac{1}{\Gamma(m-\nu)}\int_{(i-1)\tau}^{t}\frac{(t-s)^{n-1}}{(n-1)!}\int_{(i-1)\tau}^{s}(s-s_1)^{m-\nu-1}\leftidx{_{\scs{(i)}}}e_{\scs N}^{(m)}(s_1)\mathrm{d}s_1\mathrm{d}s\hspace{6cm}\nonumber\\
=\frac{1}{\Gamma(m-\nu)(n-2)!(m-\nu)}\int_{(i-1)\tau}^{t}\int_{(i-1)\tau}^{s}(t-s)^{n-2}(s-s_1)^{m-\nu}\leftidx{_{\scs{(i)}}}e_{\scs N}^{(m)}(s_1)\mathrm{d}s_1\mathrm{d}s.\hspace{4cm}\nonumber
\end{align}
By repeating the above steps, i.e., after $k$ times integrating by parts we have
\begin{align}
\frac{1}{\Gamma(m-\nu)}\int_{(i-1)\tau}^{t}\frac{(t-s)^{n-1}}{(n-1)!}\int_{(i-1)\tau}^{s}(s-s_1)^{m-\nu-1}\leftidx{_{\scs{(i)}}}e_{\scs N}^{(m)}(s_1)\mathrm{d}s_1\mathrm{d}s&\nonumber\\
=M\int_{(i-1)\tau}^{t}\int_{(i-1)\tau}^{s}(t-s)^{n-k-1}(s-s_1)^{m-\nu+k-1}\leftidx{_{\scs{(i)}}}e_{\scs N}^{(m)}(s_1)\mathrm{d}s_1\mathrm{d}s,&\nonumber
\end{align}
where
\[
M=\frac{1}{\Gamma(m-\nu)(n-k-1)!\prod_{j=1}^{k}(m-\nu+j-1)}.
\]
Now, by putting $k=n-1$ we conclude the relation~\eqref{si3}.
\section{Proof of relation~(33)}
Integrating by parts of the left-hand side of~\eqref{si5} yields
\begin{align}
\frac{1}{\Gamma(m-\nu)\prod_{j=1}^{n-1}(m-\nu+j-1)}\int_{(i-1)\tau}^{t}\int_{(i-1)\tau}^{s}(s-s_1)^{m+n-\nu-2}\leftidx{_{\scs{(i)}}}e_{\scs N}^{(m)}(s_1)&\mathrm{d}s_1\mathrm{d}s\nonumber\\
=\frac{1}{\Gamma(m-\nu)\prod_{j=1}^{n-1}(m-\nu+j-1)}\int_{(i-1)\tau}^{t}(s-s_1)^{m+n-\nu-2}\leftidx{_{\scs{(i)}}}e_{\scs N}^{(m-1)}(s_1)|_{s_1=(i-1)\tau}^{s_1=s}&\mathrm{d}s\nonumber\\
+\frac{(m+n-\nu-2)}{\Gamma(m-\nu)\prod_{j=1}^{n-1}(m-\nu+j-1)}\int_{(i-1)\tau}^{t}\int_{(i-1)\tau}^{s}(s-s_1)^{m+n-\nu-3}\leftidx{_{\scs{(i)}}}e_{\scs N}^{(m-1)}(s_1)&\mathrm{d}s_1\mathrm{d}s.\nonumber
\end{align}
After substitution of the limits of integration,  the first term of right-hand side of the above relation becomes zero, considering the initial conditions of Eq.~\eqref{aa1}. Therefore, we have
\begin{align}
\frac{1}{\Gamma(m-\nu)\prod_{j=1}^{n-1}(m-\nu+j-1)}\int_{(i-1)\tau}^{t}\int_{(i-1)\tau}^{s}(s-s_1)^{m+n-\nu-2}\leftidx{_{\scs{(i)}}}e_{\scs N}^{(m)}(s_1)&\mathrm{d}s_1\mathrm{d}s\nonumber\\
=\frac{(m+n-\nu-2)}{\Gamma(m-\nu)\prod_{j=1}^{n-1}(m-\nu+j-1)}\int_{(i-1)\tau}^{t}\int_{(i-1)\tau}^{s}(s-s_1)^{m+n-\nu-3}\leftidx{_{\scs{(i)}}}e_{\scs N}^{(m-1)}(s_1)&\mathrm{d}s_1\mathrm{d}s.\nonumber
\end{align}
By repeating the above steps, i.e., after $m-3$ times integrating by parts we have
\begin{align}
\frac{1}{\Gamma(m-\nu)\prod_{j=1}^{n-1}(m-\nu+j-1)}\int_{(i-1)\tau}^{t}\int_{(i-1)\tau}^{s}(s-s_1)^{m+n-\nu-2}\leftidx{_{\scs{(i)}}}e_{\scs N}^{(m)}(s_1)&\mathrm{d}s_1\mathrm{d}s\nonumber\\
=\frac{\prod_{j=1}^{m-3}(m+n-\nu-1-j)}{\Gamma(m-\nu)\prod_{j=1}^{n-1}(m-\nu+j-1)}\int_{(i-1)\tau}^{t}\int_{(i-1)\tau}^{s}(s-s_1)^{n-\nu+1}\leftidx{_{\scs{(i)}}}e_{\scs N}^{(3)}(s_1)&\mathrm{d}s_1\mathrm{d}s.\nonumber
\end{align}
The obtained result completes the proof of relation~\eqref{si5}.
\section{Boundedness of $D^\nu$}\label{Ap3}
By the definition of Caputo fractional derivative, we have
\begin{equation}
D^\nu u_{\scs N}(t)=\frac{1}{\Gamma(m-\nu)}\int_{0}^{t}(t-s)^{m-\nu-1}u_{\scs N}^{(m)}(s)ds.
\end{equation}
As we know,  if $u_{\scs N}^{(m)}:[0, t] \rightarrow\mathbb{R}$ is continuous and $(t-s)^{m-\nu-1}$ is an integrable function that does not change sign on $[0, t]$, then there exists $\xi$ in $(0, t)$ such that
\[
\int_{0}^{t}(t-s)^{m-\nu-1}u_{\scs N}^{(m)}(t)ds=u_{\scs N}^{(m)}(\xi)\int_{0}^{t}(t-s)^{m-\nu-1}ds=\frac{u_{\scs N}^{(m)}(\xi)}{m-\nu}t^{m-\nu}.
 \]
 Therefore,
 \[
 |D^\nu u_{\scs N}(t)|\le M|u_{\scs N}^{(m)}(\xi)|\le M|u_{\scs N}|_{H^{1:N}_w}.
 \]
 Hence, for $k\ge 1$ we have
 \[
|D^\nu u_{\scs N}|_{H^{k:N}_w}\leq C| u_{\scs N}|_{H^{k:N}_w},
 \]
 which means $D^\nu$ is bounded.

\bibliographystyle{amsplain}
\bibliography{refdata1}%
\end{document}